\documentclass[a4paper,12pt]{amsart}

\usepackage[T1]{fontenc}
\usepackage[utf8]{inputenc}
\usepackage{lmodern}
\usepackage[english]{babel}
\usepackage[autostyle]{csquotes}

\usepackage[hidelinks,pdfusetitle]{hyperref}
\usepackage{doi}

\usepackage{enumitem}

\usepackage{amsmath,amssymb,amsthm,amscd}
\usepackage{tikz-cd}

\newtheorem{thm}{Theorem}[section]
\newtheorem{prop}[thm]{Proposition}
\newtheorem{lem}[thm]{Lemma}
\newtheorem{cor}[thm]{Corollary}

\theoremstyle{definition}

\newtheorem{defn}[thm]{Definition}
\newtheorem{defprop}[thm]{Definition/Proposition}

\theoremstyle{remark}
\newtheorem{rem}[thm]{Remark}

\begin{document}
\title[$p$-approximation properties of discrete groups]{On some $p$-approximation properties of discrete groups and $\ell^p$ uniform Roe algebras}
\author{Yeong Chyuan Chung}
\address{School of Mathematics, Jilin University, Changchun 130012, Jilin, P. R. China}
\email{chungyc@jlu.edu.cn}
\date{\today}

\begin{abstract}
We study approximation properties of discrete groups
and $\ell^p$ uniform Roe algebras.
Property A implies $p$-nuclearity of uniform Roe
algebras and exactness of the coefficient uniform
Roe space functor, with these conclusions holding
unconditionally at $p=1$.
For $1<p<\infty$, we characterize the
$p$-approximation property and $p$-weak amenability
through approximation properties of the inclusion
of the reduced $L^p$ group algebra into the algebra
of $p$-convoluters.
We study the $p$-invariant translation approximation
property and its operator space version, obtaining
further characterizations of the $p$-approximation
property through invariant uniform Roe algebras
and, for exact groups, fixed-point tensor product
identities.
We also establish permanence results for invariant
translation approximation properties and show that
the $p$-approximation property is preserved under
group extensions.
\end{abstract}

\maketitle

\tableofcontents

\section{Introduction}

Approximation properties of operator algebras provide an important link between coarse geometry, harmonic analysis, and geometric group theory. For a discrete metric space with bounded geometry, Yu's property A is equivalent to nuclearity of its uniform Roe algebra \cite[Theorem~5.5.7]{BO}. For countable discrete groups, property A is in turn equivalent to exactness of the reduced group $C^*$-algebra \cite{Ozawa,GK}. These connections motivate the study of analogous approximation properties in the $L^p$ setting.

The theory of $p$-operator spaces \cite{Pisier,LM96,LM96i,Daws10,ALR} provides a framework for studying approximation properties beyond the Hilbert space setting. The principal algebras considered here are the
reduced $L^p$ operator algebra $F_\lambda^p(\Gamma)$ of a discrete group $\Gamma$, also known as the algebra of $p$-pseudofunctions, and the $\ell^p$ uniform Roe algebra $B_u^p(X)$ of a bounded geometry metric space $X$. We work primarily with $1<p<\infty$, while also examining the endpoint $p=1$ for uniform Roe algebras and their coefficient spaces.

For a countable discrete group equipped with a proper right-invariant metric, the reduced group algebra is naturally contained in the uniform Roe algebra. In the case $p=2$, the question of whether this inclusion identifies the reduced group algebra with the fixed-point algebra under right translation is known as the invariant translation approximation property \cite[Section~11.5.3]{RoeL}. Its connections with the approximation property of Haagerup--Kraus were studied by Zacharias \cite{Zach}, and the corresponding fixed-point tensor product identities were studied by Katsura and Uuye \cite{KU}. We investigate these questions in the $L^p$ setting, drawing also on the characterization of the $p$-approximation property by An, Lee and Ruan \cite{ALR}.

We first show that property A implies $p$-nuclearity of $B_u^p(X)$ (Theorem~\ref{Apnuc}). Under the same hypothesis, for every $SQ_p$ space $E$ and every closed $p$-operator space $S\subseteq B(E)$, we have
\[
B_u^p(X,S)
=
\{a\in B_u^p(X,B(E)):
a_{x,y}\in S\text{ for all }x,y\in X\}.
\]
Moreover, the coefficient functor $B_u^p(X,-)$ preserves short exact sequences of concrete $p$-operator spaces, including their $p$-complete quotient structure (Theorem~\ref{thm:roe-coefficients}). At $p=1$, all these conclusions hold without property A.
As an application, we show that each of the two coefficient properties above, when required for all $p\in(1,\infty)$, characterizes exactness of countable discrete groups (Theorem~\ref{thm:functor}).

Turning to approximation properties of discrete groups, we extend the characterization of the $p$-approximation property ($p$-AP) in \cite{ALR} by allowing finite-rank approximating maps to take values in the algebra of $p$-convoluters. More precisely, a discrete group $\Gamma$ has the $p$-AP if and only if the canonical inclusion
\[
F_\lambda^p(\Gamma)\longrightarrow CV_p(\Gamma)
\]
has the $p$-operator approximation property ($p$-OAP) (Theorem~\ref{thm:convolver-inclusion-oap}). Likewise, $p$-weak amenability is equivalent to the existence of uniformly $p$-completely bounded finite-rank point-norm approximations to this inclusion, with equality of the optimal constants (Proposition~\ref{prop:convolver-inclusion-cbap}).

For a countable discrete group $\Gamma$, write $|\Gamma|$ for its metric realization with respect to a proper right-invariant metric. The $p$-invariant translation approximation property ($p$-ITAP) asserts that
\[
F_\lambda^p(\Gamma)=B_u^p(|\Gamma|)^\Gamma.
\]
We also study the coefficient version
\[
F_\lambda^p(\Gamma)\otimes S
=B_u^p(|\Gamma|,S)^\Gamma
\]
for concrete $p$-operator spaces $S$. 
Requiring this equality for every closed $p$-operator subspace $S\subseteq K(\ell^p)$ gives the $p$-operator ITAP.

We establish the scalar $p$-ITAP for amenable groups and obtain sufficient conditions in terms of polynomial and Gevrey rapid decay with respect to length functions of negative type (Corollary~\ref{amenable} and Propositions~\ref{prop:RDp} and~\ref{prop:GRD-ITAP}). The coefficient version passes to subgroups and finite-index supergroups (Propositions~\ref{prop:subgp} and \ref{prop:finite-index}). At $p=1$, the coefficient fixed-point identity holds for every countable discrete group and every closed $1$-operator space on a Banach space (Proposition~\ref{prop:1ITAP}).

For $1<p<\infty$, the $p$-AP implies both the coefficient-entry identity above with $X=|\Gamma|$ and the equality
\[
B_u^p(|\Gamma|,S)^\Gamma
=(B_u^p(|\Gamma|)\otimes S)^\Gamma
=F_\lambda^p(\Gamma)\otimes S
\]
for every concrete $p$-operator space $S$ (Theorem~\ref{thm:pAPcoeff}). Conversely, the $p$-operator ITAP together with the coefficient-entry identity for closed subspaces of $K(\ell^p)$ implies the $p$-AP (Theorem~\ref{thm:operator-itap-entry}).

The convolver inclusion results give a further characterization: $\Gamma$ has the $p$-AP if and only if $B_u^p(|\Gamma|)^\Gamma$ has the $p$-OAP, and $\Gamma$ is $p$-weakly amenable if and only if this invariant uniform Roe algebra has the $p$-CBAP, with equality of the optimal constants (Theorem~\ref{thm:invariant-roe-oap}).

For exact groups, the $p$-AP is equivalent to the $p$-operator ITAP, and also to the fixed-point tensor identity
\[
B_u^p(|\Gamma|)^\Gamma\otimes S
=(B_u^p(|\Gamma|)\otimes S)^\Gamma
\]
for every closed $p$-operator subspace $S\subseteq K(\ell^p)$ (Theorem~\ref{thm:exact-group-characterizations}).
This extends the classical results of Zacharias and Katsura--Uuye \cite{Zach,KU}. By Vergara's results \cite{Ver}, higher-rank lattices of the type considered in Corollary~\ref{cor:higher-rank} fail both the $p$-operator ITAP and the fixed-point tensor identity for every $p\in(1,\infty)$.

Finally, we study invariant translation approximation properties for direct products and group extensions.
For a short exact sequence of countable discrete groups
\[
1\longrightarrow H\longrightarrow\Gamma
\longrightarrow Q\longrightarrow 1,
\]
if $Q$ has the $p$-AP, then $\Gamma$ has the $p$-ITAP for a coefficient space $S$ if and only if $H$ does
(Corollary~\ref{cor:extension-equivalence}). We also establish permanence of the $p$-AP itself under group extensions for arbitrary discrete groups (Theorem~\ref{cor:pAP-extension}).

Section~2 recalls the necessary preliminaries and develops the multiplier approximation results.
Section~3 studies approximation and coefficient properties of uniform Roe algebras.
Section~4 investigates the $p$-ITAP and its permanence properties.
Section~5 establishes the necessary results on Fubini products and slice maps.
Section~6 relates the $p$-AP to approximation properties of the convolver inclusion and invariant uniform Roe algebras.
Section~7 studies direct products and group extensions.

\subsection*{AI Declaration}

OpenAI's ChatGPT (GPT-5.6/6/6.2) was used to assist with mathematical exploration, proof development and verification, literature searches, and manuscript preparation, particularly in the study of approximation properties of the convolver inclusion and group extensions. All mathematical arguments and references arising from such assistance were independently verified by the author, who takes full responsibility for the content of the article.

\section{Preliminaries}

Unless otherwise specified, we assume that
$p\in(1,\infty)$.
The definitions of concrete $p$-operator spaces,
$p$-completely bounded maps, uniform
Roe algebras and their coefficient spaces, and the $p$-CBAP and $p$-CCAP will also
be used at $p=1$, with the same matrix norm
conventions.
We write $\ell^p$ for $\ell^p(\mathbb N)$ and
$\ell_n^p$ for the $n$-dimensional $\ell^p$ space.
Duality pairings between Banach spaces and their
duals are taken to be bilinear.
Concrete $p$-operator spaces are understood to be
norm-closed linear subspaces with their inherited
matrix norms.

\subsection{$p$-operator spaces}

\begin{defn} \cite{ALR,Daws10}
For $p\in(1,\infty)$, an abstract $p$-operator space is a Banach space $X$ together with a family of norms $||\cdot||_n$ on $M_n(X)$, with $\|\cdot\|_1$ being the given norm on $X$, satisfying:
\begin{enumerate}
\item[$\mathcal{D}_\infty$:] For $u\in M_n(X)$ and $v\in M_m(X)$, we have \[\biggl\Vert\begin{pmatrix} u & 0 \\ 0 & v \end{pmatrix}\biggr\Vert_{n+m}=\max(||u||_n,||v||_m);\]
\item[$\mathcal{M}_p$:] For $u\in M_m(X)$, $\alpha\in M_{n,m}(\mathbb{C})$, and $\beta\in M_{m,n}(\mathbb{C})$, we have \[||\alpha u\beta||_n\leq ||\alpha||_{B(\ell_m^p,\ell_n^p)} ||u||_m ||\beta||_{B(\ell_n^p,\ell_m^p)}.\]
\end{enumerate}
\end{defn}

\begin{defn}
Let $V$ and $W$ be $p$-operator spaces. A linear map $\phi:V\to W$ is said to be $p$-completely bounded if
\[ \|\phi\|_{pcb} = \sup_n \|\phi_n\|<\infty, \]
where $\phi_n:M_n(V)\to M_n(W)$ is the map sending $[x_{ij}]\in M_n(V)$ to $[\phi(x_{ij})]\in M_n(W)$.

The map $\phi$ is said to be a $p$-complete contraction (resp. $p$-complete isometry) if $\|\phi\|_{pcb}\leq 1$ (resp. $\phi_n$ is an isometry for each $n$).

The space of all $p$-completely bounded maps from $V$ to $W$ is denoted by $CB_p(V,W)$.
\end{defn}

In the general theory of $p$-operator spaces, one typically considers only $p\in(1,\infty)$ but the definition still makes sense when $p=1$. Le Merdy showed in \cite[Theorem 4.1]{LM96} that for $p\in(1,\infty)$, an abstract $p$-operator space can be $p$-completely isometrically embedded in $B(E)$ for some $E\in SQ_p$, where $SQ_p$ denotes the collection of subspaces of quotients (equivalently, quotients of subspaces) of $L^p$ spaces, and $B(E)$ denotes the collection of all bounded linear operators on $E$. 
The class $SQ_p$ is exactly the class of $p$-spaces studied by Herz \cite{Herz}. This follows from \cite[Theorem 2']{Kwa72} (see \cite[Sections 1.8-1.9]{multinorm} for a detailed proof).
The class $SQ_2$ is the class of Hilbert spaces while the class $SQ_1$ contains all Banach spaces since every Banach space is a quotient of some $L^1$ space.

From Le Merdy's characterization, one can see that $CB_p(V,W)$ is a $p$-operator space with matrix norms given by identifying $M_n(CB_p(V,W))$ with $CB_p(V,M_n(W))$.

Following \cite{ALR}, we refer to $p$-operator
subspaces of $B(L^p(\mu))$ as $p$-operator spaces
on $L^p$ spaces. We will use this class of spaces
in particular when working with injective tensor
products and slice maps.
For $p\in[1,\infty)$, we refer to norm-closed
subalgebras of $B(E)$ as $SQ_p$-operator algebras
when $E$ is an $SQ_p$ space, and as $L^p$-operator
algebras when $E$ is an $L^p$ space.

Given a measure $\mu$ and a Banach space $E$, one can define a norm on the algebraic tensor product $L^p(\mu)\odot E$ by embedding it into the space $L^p(\mu,E)$ of Bochner $p$-integrable functions via the map $f\otimes x\mapsto f(\cdot)x$.
The completion is denoted by $L^p(\mu)\otimes_p E$. We refer the reader to \cite[Chapter 7]{DF} for details regarding this tensor product.
If $E,F$ are $SQ_p$ spaces, $S\in B(E,F)$, and $T\in B(L^p(\mu))$, then the operator $T\otimes S$ is bounded as an operator from $L^p(\mu)\otimes_p E$ to $L^p(\mu)\otimes_p F$ with norm $\|T\| \|S\|$.
Thus, given $p$-operator spaces $V\subseteq B(L^p(\mu))$ and $W\subseteq B(E)$, we may consider their spatial tensor product $V\otimes W$, which is the norm closure of the algebraic tensor product $V\odot W$ in $B(L^p(\mu)\otimes_pE)$.

There is also the $p$-operator space injective tensor product introduced in \cite{ALR}, which we now describe.
Given a $p$-operator space $V$, its dual space $V'=CB_p(V,\mathbb{C})$ has a natural $p$-operator space structure given by identifying $M_n(V')$ with $CB_p(V,B(\ell^p_n))$, where $\ell^p_n$ denotes $\mathbb{C}^n$ equipped with the $\ell^p$ norm.
In fact, for any $p$-operator space $V$, its dual $V'$ is always a $p$-operator space on $L^p$ space \cite[Theorem 4.3]{Daws10}.
Given two $p$-operator spaces $V$ and $W$, there is an injective map from the algebraic tensor product $V\odot W$ to $CB_p(V',W)$ given by $v\otimes w\mapsto(f\mapsto f(v)w)$ for $v\in V$ and $w\in W$.
The completion of $V\odot W$ in $CB_p(V',W)$ is denoted by $V\stackrel{\vee_p}{\otimes}W$, and is called the $p$-operator space injective tensor product of $V$ and $W$.

If both $V$ and $W$ are $p$-operator spaces on $L^p$ space, then their spatial tensor product and their $p$-operator space injective tensor product coincide.
This is because if $W\subseteq B(L^p(\nu))$, then the canonical inclusion $B(L^p(\mu))\stackrel{\vee_p}{\otimes}W\hookrightarrow B(L^p(\mu)\otimes_p L^p(\nu))$ is a $p$-complete isometry \cite[Theorem 3.2]{ALR}, and 
also the inclusion $V\stackrel{\vee_p}{\otimes}W\hookrightarrow B(L^p(\mu))\stackrel{\vee_p}{\otimes}W$ is a $p$-complete isometry \cite[Proposition 2.3.6]{Lee}. 

In particular, for $p$-operator spaces on $L^p$ spaces, we may use the injectivity and functoriality of the $p$-operator space injective tensor product for the spatial tensor product. Thus inclusions of $p$-operator subspaces induce $p$-complete isometries on the tensor products. Moreover, if $\varphi:V\to V_1$ and $\psi:W\to W_1$ are $p$-completely bounded maps between such spaces, then $\varphi\otimes\psi$ extends to a $p$-completely bounded map satisfying
\[
\|\varphi\otimes\psi\|_{pcb}
\leq\|\varphi\|_{pcb}\|\psi\|_{pcb}.
\]

\subsection{$\ell^p$ uniform Roe algebras}

\begin{defn}
Let $X$ be a metric space. Then $X$ is said to have \emph{bounded geometry} if for all $R\geq 0$ there exists $N_R\in\mathbb{N}$ such that for all $x\in X$, the ball of radius $R$ about $x$ has at most $N_R$ elements.
\end{defn}

Note that every metric space with bounded geometry is necessarily countable and discrete.

\begin{defn} 
For an operator $T=[T_{xy}]_{x,y\in X}\in B(\ell^p(X))$, where $T_{xy}=\delta_x^*(T\delta_y)$, define the propagation of $T$ to be
\[ \mathop{\rm prop}(T)=\sup\{ d(x,y):x,y\in X,T_{xy}\neq 0 \}\in[0,\infty]. \]
Denote by $\mathbb{C}_u^p[X]$ the unital algebra of all bounded operators on $\ell^p(X)$ with finite propagation. The $\ell^p$ uniform Roe algebra, denoted by $B^p_u(X)$, is defined to be the operator norm closure of $\mathbb{C}_u^p[X]$ in $B(\ell^p(X))$.
\end{defn}

Let $X$ be a metric space with bounded geometry.
For any subset $S\subseteq B(E)$, where $E$ is an $SQ_p$ space, let $\mathbb{C}[X,S]$ denote the set of all finite propagation matrices $[a_{x,y}]_{x,y\in X}$, where $a_{x,y}\in S$ and $\sup_{x,y\in X}\Vert a_{x,y} \Vert<\infty$. Then we have a canonical inclusion $\mathbb{C}[X,S]\subseteq B(\ell^p(X,E))$.
Note that $\ell^p(X,E)\cong \ell^p(X)\otimes_p E$ is still an $SQ_p$ space.

If $a\in\mathbb C[X,S]$ has propagation at most $R$,
then
\[
\|a\|\leq N_R\sup_{x,y\in X}\|a_{x,y}\|,
\]
where $N_R$ is a uniform bound on the cardinalities
of balls of radius $R$. This follows from the scalar
$\ell^p$ estimate for matrices with at most $N_R$
nonzero entries in each row and column.

\begin{defn}
Given a $p$-operator space $S\subseteq B(E)$, define $B^p_u(X,S)$ as the closure of $\mathbb{C}[X,S]$ in $B(\ell^p(X,E))$.
\end{defn}

If $S\subseteq B(E)$ and $T\subseteq B(F)$ are $p$-operator spaces
and $\theta:S\to T$ is $p$-completely bounded, then entrywise
application of $\theta$ defines a map
\[
\widehat{\theta}:\mathbb{C}[X,S]\longrightarrow \mathbb{C}[X,T],
\qquad
\widehat{\theta}([a_{x,y}])=[\theta(a_{x,y})].
\]
This map extends uniquely to a $p$-completely bounded map
\[
\widehat{\theta}:B_u^p(X,S)\longrightarrow B_u^p(X,T)
\]
satisfying
\[
\|\widehat{\theta}\|_{pcb}
\leq \|\theta\|_{\mathrm{pcb}}.
\]
Indeed, if $a\in \mathbb{C}[X,S]$, then
\[
\|a\|
=
\sup_{K\Subset X}\|P_Ka\iota_K\|,
\]
where $K\Subset X$ means that $K$ is finite, $P_K$ is the canonical
projection, and $\iota_K$ is the canonical inclusion. Hence
\[
\|P_K\widehat{\theta}(a)\iota_K\|
=
\|\theta_{|K|}(P_Ka\iota_K)\|
\leq
\|\theta\|_{\mathrm{pcb}}\|P_Ka\iota_K\|.
\]
Applying the same argument at each matrix level,
using the canonical isometric identification
\[
M_n(B_u^p(X,S))\cong B_u^p(X,M_n(S)),
\]
gives the asserted $p$-completely bounded norm
estimate.
If $\theta$ is a homomorphism, then so is $\widehat{\theta}$.

A surjective linear map $\theta:V\to W$ between
$p$-operator spaces is called a $p$-complete quotient
map if the induced map $V/\ker\theta\to W$ is a
$p$-complete isometry, where the quotient is equipped
with the matrix quotient norms.
When we speak of a short exact sequence
\[
0\longrightarrow J\longrightarrow A
\overset{\pi}{\longrightarrow}B\longrightarrow0
\]
of concrete $p$-operator spaces, we understand $J$ to have its inherited
$p$-operator structure and $B$ to have the quotient $p$-operator
structure; equivalently, the induced map $A/J\to B$ is a
$p$-complete isometry. In particular, $\pi$ is a $p$-complete
quotient map.

If, in addition, $J$, $A$, and $B$ are $SQ_p$-operator algebras and
the maps are homomorphisms, we refer to the above sequence as a short
exact sequence of $SQ_p$-operator algebras.

For a linear subspace $S\subseteq B(E)$,
where $E$ is an $SQ_p$ space, the propagation
estimate above gives
\[
\mathbb C[X,\overline S]
\subseteq \overline{\mathbb C[X,S]}.
\]
Indeed, the entries of any element of
$\mathbb C[X,\overline S]$ of propagation at most
$R$ can be approximated uniformly by elements
of $S$, without increasing propagation.

Moreover, if $a\in\mathbb C_u^p[X]$ and $b\in S$,
then $a\otimes b$ belongs to $\mathbb C[X,S]$.
Consequently, for every closed $p$-operator
space $S\subseteq B(E)$, we have the canonical
inclusion
\[
B_u^p(X)\otimes S\subseteq B_u^p(X,S)
\]
in $B(\ell^p(X,E))$
(cf.\ \cite[Lemmas~3.4 and~3.6]{UZ}).

Given two metric spaces $(X,d_X)$ and $(Y,d_Y)$, we equip the product $X\times Y$ with the metric
\[ d_{X\times Y}((x_1,y_1),(x_2,y_2)) = \max\{ d_X(x_1,x_2),d_Y(y_1,y_2) \}. \]
If $X$ and $Y$ are of bounded geometry, then so is $X\times Y$.

\begin{prop} \label{thm:incl}
(cf.\ \cite[Theorem~3.9]{UZ})
Let $X$ and $Y$ be metric spaces with bounded
geometry, and let $S\subseteq B(E)$ be a
$p$-operator space, where $E$ is an $SQ_p$ space and $p\in[1,\infty)$.
Under the canonical isometric identification
\[
\ell^p(X\times Y,E)
\cong \ell^p(X,\ell^p(Y,E)),
\]
we have the inclusions
\[
B_u^p(X)\otimes B_u^p(Y,S)
\subseteq B_u^p(X\times Y,S)
\subseteq B_u^p(X,B_u^p(Y,S)).
\]
\end{prop}

\begin{proof}
For $a\in\mathbb C_u^p[X]$ and
$b\in\mathbb C[Y,S]$, the operator $a\otimes b$
has matrix entries
\[
(a\otimes b)_{(x,y),(x',y')}
=a_{x,x'}b_{y,y'}.
\]
It therefore has propagation at most
$\max\{\operatorname{prop}(a),\operatorname{prop}(b)\}$
and has uniformly bounded entries in $S$.
Thus $a\otimes b\in\mathbb C[X\times Y,S]$,
and the first inclusion follows by taking linear
spans and norm closures.

For the second inclusion, let
$c\in\mathbb C[X\times Y,S]$ have propagation
at most $R$. Regard $c$ as a matrix indexed by
$X\times X$ with entries
\[
c_{x,x'}=
[c_{(x,y),(x',y')}]_{y,y'\in Y}.
\]
Each $c_{x,x'}$ belongs to $\mathbb C[Y,S]$,
has propagation at most $R$, and satisfies
\[
\|c_{x,x'}\|\leq\|c\|,
\]
since it is a compression of $c$ between the
corresponding coordinate summands.
Moreover, $c_{x,x'}=0$ whenever $d_X(x,x')>R$.
Consequently,
\[
c\in\mathbb C[X,B_u^p(Y,S)].
\]
Taking norm closures gives the second inclusion.
\end{proof}

\subsection{The $p$-approximation property ($p$-AP) for discrete groups}

Let $\Gamma$ be a discrete group.
For $p\in[1,\infty)$, the left and right regular
representations $\lambda_p$ and $\rho_p$ of
$\Gamma$ on $\ell^p(\Gamma)$ are given by
\[
(\lambda_p(s)f)(t)=f(s^{-1}t),
\qquad
(\rho_p(s)f)(t)=f(ts).
\]
When the exponent is understood, we also write
$\lambda_s$ and $\rho_s$ for these operators.

The reduced $L^p$ operator algebra
$F_\lambda^p(\Gamma)$ is the norm-closed linear
span of $\lambda_p(\Gamma)$ in $B(\ell^p(\Gamma))$.
It is also known as the algebra of $p$-pseudofunctions.
For $1<p<\infty$, we write $PM_p(\Gamma)$ for
the weak$^*$-closed linear span of $\lambda_p(\Gamma)$
and $CV_p(\Gamma)=\rho_p(\Gamma)'$ for the algebra
of $p$-convoluters.

Every countable discrete group $\Gamma$ admits a proper
right-invariant metric compatible with its discrete
topology. Such a metric has finite balls. We denote the resulting bounded geometry
metric space by $|\Gamma|$.

For a function $f$ on $\Gamma$, write $\check f(s)=f(s^{-1})$.
Define $A_p(\Gamma)$ to be the space of $f\in C_0(\Gamma)$ for which there exist sequences $(\xi_n)\subseteq\ell^{p'}(\Gamma)$ and $(\eta_n)\subseteq\ell^p(\Gamma)$, where $1/p+1/p'=1$, satisfying
\[
\sum_n\|\xi_n\|_{p'}\|\eta_n\|_p<\infty
\]
and
\[ f(s)= \sum_n\xi_n * \check{\eta}_n(s) =\sum_n\langle \xi_n,\lambda_p(s)\eta_n \rangle \] with norm 
\[ \| f \|_{A_p(\Gamma)} = \inf\left\{ \sum_n \|\xi_n\| \|\eta_n\|: f=\sum_n\xi_n * \check{\eta}_n \right\}. \]
Then $A_p(\Gamma)$ is a commutative Banach algebra known as the Fig\`{a}-Talamanca-Herz algebra of $\Gamma$ \cite{F-T,Herz}.
Daws \cite{Daws10} showed that $A_p(\Gamma)$ carries a natural $p$-operator space structure.
We use the dual $p$-operator space structure on $A_p(\Gamma)$ as in \cite{ALR}, so that $A_p(\Gamma)^*$ is
$p$-completely isometrically identified with the algebra $PM_p(\Gamma)$ of $p$-pseudomeasures.

A function $u:\Gamma\to\mathbb{C}$ is called a multiplier of $A_p(\Gamma)$ if the multiplication map
\[ m_u:\phi\in A_p(\Gamma) \to u\phi\in A_p(\Gamma) \]
is well-defined, and thus bounded, on $A_p(\Gamma)$.
In this case, $u$ is a bounded (continuous) function on $\Gamma$.
Let $M_{cb}A_p(\Gamma)$ denote the space of $p$-completely bounded multipliers of $A_p(\Gamma)$ with norm $\| u \|_{M_{cb}A_p(\Gamma)} = \| m_u \|_{pcb}$.

It is known that $M_{cb}A_p(\Gamma)$ is a dual space with canonical predual (isometrically isomorphic to) $Q_{pcb}(\Gamma)$, which we describe below.
The $p=2$ case was due to de Canni\`{e}re and Haagerup \cite{CH}, while the general case was attributed by An-Lee-Ruan \cite{ALR} to private communication with Miao (also see \cite{Miao}).
Given $f\in\ell^1(\Gamma)$, define a bounded linear functional $\alpha_{cb}(f)$ on $M_{cb}A_p(\Gamma)$ by
\[ \alpha_{cb}(f)(u)=\sum_{s\in\Gamma}f(s)u(s). \]
Each $u\in M_{cb}A_p(\Gamma)$ is bounded on $\Gamma$ with $\sup_{s\in\Gamma}|u(s)|\leq\| u \|_{M_{cb}A_p(\Gamma)}$, so $\alpha_{cb}:\ell^1(\Gamma)\to M_{cb}A_p(\Gamma)^*$ is an injective contraction.
The norm closure of $\alpha_{cb}(\ell^1(\Gamma))$ in $M_{cb}A_p(\Gamma)^*$ is defined to be $Q_{pcb}(\Gamma)$.

Based on this duality, An-Lee-Ruan introduced the following generalization of the (2-)approximation property of Haagerup and Kraus \cite{HK}.

\begin{defn} (cf.\ \cite{ALR,Ver})
A discrete group $\Gamma$ has the $p$-approximation property ($p$-AP) if there is a net $\{u_\alpha\}$ in $A_p(\Gamma)$ such that $u_\alpha\to 1$ in the $\sigma(M_{cb}A_p(\Gamma),Q_{pcb}(\Gamma))$ topology (i.e., the weak* topology on $M_{cb}A_p(\Gamma)$). 
\end{defn}

When dealing with the $p$-AP, one can take the net $\{u_\alpha\}$ to consist of finitely supported functions on $\Gamma$ \cite[Remark 1.2]{Ver}.

Following An--Lee--Ruan \cite{ALR}, a discrete group
$\Gamma$ is said to be $p$-weakly amenable if
$A_p(\Gamma)$ has an approximate identity $(u_\alpha)$
such that
\[
    \|u_\alpha\|_{M_{cb}A_p(\Gamma)}\leq C
\]
for some constant $C\geq 1$ and all $\alpha$.
We denote the infimum of all such constants $C$ by
$\Lambda_p(\Gamma)$, with $\Lambda_p(\Gamma)=\infty$
if $\Gamma$ is not $p$-weakly amenable.

\begin{lem} \label{lem:vector-valued-Schur}
Let $p\in(1,\infty)$, let $\Gamma$ be a countable
discrete group, let $G$ be a discrete group, and
let $\pi:\Gamma\to G$ be a group homomorphism.
Let $E$ be an $SQ_p$ space and let $S\subseteq B(E)$
be a $p$-operator space.

For every $u\in M_{cb}A_p(G)$, the formula
\[
\widehat m_u^\pi([x_{s,t}])
=
[u(\pi(st^{-1}))x_{s,t}]
\]
defines a $p$-completely bounded map
\[
\widehat m_u^\pi:B_u^p(|\Gamma|,S)
\longrightarrow B_u^p(|\Gamma|,S)
\]
satisfying
\[
\|\widehat m_u^\pi\|_{pcb}
\leq\|u\|_{M_{cb}A_p(G)}.
\]
When $\pi=\mathrm{id}_\Gamma$, we write
$\widehat m_u$ for $\widehat m_u^\pi$.
\end{lem}

\begin{proof}
By \cite[Proposition~8.4 and Theorem~8.6]{Daws10},
for every $\varepsilon>0$ there are a measure space
$(\Omega,\mu)$ and maps
\[
\alpha:G\longrightarrow L^p(\Omega),\qquad
\beta:G\longrightarrow L^q(\Omega),
\]
where $1/p+1/q=1$, such that
\[
u(gh^{-1})=\langle\beta(g),\alpha(h)\rangle
\]
for all $g,h\in G$, and
\[
\sup_{h\in G}\|\alpha(h)\|_p
\sup_{g\in G}\|\beta(g)\|_q
<
\|u\|_{M_{cb}A_p(G)}+\varepsilon.
\]

Define
\[
V:\ell^p(\Gamma,E)
\longrightarrow L^p(\Omega;\ell^p(\Gamma,E))
\]
and
\[
U:L^p(\Omega;\ell^p(\Gamma,E))
\longrightarrow\ell^p(\Gamma,E)
\]
by
\[
(V\xi)(\omega)_t
=\alpha(\pi(t))(\omega)\xi_t
\]
and
\[
(U\eta)_s
=\int_\Omega
\beta(\pi(s))(\omega)\eta(\omega)_s\,d\mu(\omega).
\]
Then
\[
\|V\|\leq\sup_{h\in G}\|\alpha(h)\|_p,\qquad
\|U\|\leq\sup_{g\in G}\|\beta(g)\|_q.
\]

For $T\in B(\ell^p(\Gamma,E))$, let $\widetilde T$
denote its pointwise amplification to
$L^p(\Omega;\ell^p(\Gamma,E))$.
Since $\pi$ is a homomorphism, we have
\[
(U\widetilde T V)_{s,t}
=
\langle\beta(\pi(s)),\alpha(\pi(t))\rangle T_{s,t}
=
u(\pi(st^{-1}))T_{s,t}.
\]
Consequently,
\[
\|\widehat m_u^\pi(T)\|
\leq
(\|u\|_{M_{cb}A_p(G)}+\varepsilon)\|T\|.
\]
Letting $\varepsilon\to 0$ gives the asserted
operator norm estimate.

The same argument with $E$ replaced by $\ell_n^p(E)$
gives the estimate at every matrix level.
Since Schur multiplication does not increase
propagation and preserves the coefficient space $S$,
the map restricts to $B_u^p(|\Gamma|,S)$.
\end{proof}

Since the kernel
$(s,t)\mapsto u(\pi(st^{-1}))$ is invariant under
the right diagonal action of $\Gamma$, the map
$\widehat m_u^\pi$ preserves the $\Gamma$-fixed
subspace whenever the right translation action
is considered.

\begin{lem} (cf. \cite[Lemma 2.4]{Zach}) \label{lem:Qpcb}
Let $\pi:\Gamma\to G$ be a group homomorphism, where
$\Gamma$ is a countable discrete group and $G$ is a
discrete group. Let $E$ be an $SQ_p$ space and let
$S\subseteq B(E)$ be a $p$-operator space.

For every $\varphi\in B_u^p(|\Gamma|,S)^*$ and
$x\in B_u^p(|\Gamma|,S)$, the functional
\[
\omega_{\varphi,x}^\pi(u)
=
\varphi(\widehat m_u^\pi(x)),
\qquad u\in M_{cb}A_p(G),
\]
belongs to $Q_{pcb}(G)$.
When $\pi=\mathrm{id}_\Gamma$, we write
$\omega_{\varphi,x}$ for $\omega_{\varphi,x}^\pi$.
\end{lem}

\begin{proof}
By the preceding lemma,
\[
\|\omega_{\varphi,x}^\pi\|
\leq\|\varphi\|\|x\|.
\]
Since $Q_{pcb}(G)$ is norm closed in
$M_{cb}A_p(G)^*$, it suffices to consider the case
where $x$ has finite propagation.

In this case, the set
\[
D=\{\pi(st^{-1}):x_{s,t}\neq 0\}
\]
is finite. Hence $\omega_{\varphi,x}^\pi(u)$
depends only on the restriction of $u$ to $D$.
By linearity, there exist scalars $(c_g)_{g\in D}$
such that
\[
\omega_{\varphi,x}^\pi(u)
=\sum_{g\in D}c_gu(g).
\]
By the definition of $Q_{pcb}(G)$, every evaluation
functional $u\mapsto u(g)$ belongs to $Q_{pcb}(G)$.
Hence $\omega_{\varphi,x}^\pi\in Q_{pcb}(G)$.
\end{proof}

\begin{cor}
\label{cor:multiplier-transfer}
Let $\pi:\Gamma\to G$ be a group homomorphism,
where $\Gamma$ is a countable discrete group and
$G$ is a discrete group with the $p$-AP.
Let $E$ be an $SQ_p$ space and let
$S\subseteq B(E)$ be a $p$-operator space.

Then there exists a net $(v_\alpha)$ of finitely
supported elements of $A_p(G)$ such that
\[
\widehat m_{v_\alpha}^\pi(x)\longrightarrow x
\]
in norm for every $x\in B_u^p(|\Gamma|,S)$.
\end{cor}

\begin{proof}
Since $G$ has the $p$-AP, there exists a net
$(u_\alpha)$ of finitely supported elements of
$A_p(G)$ converging to $1$ in the
$\sigma(M_{cb}A_p(G),Q_{pcb}(G))$ topology.

By the preceding lemma, for every
$\varphi\in B_u^p(|\Gamma|,S)^*$ and
$x\in B_u^p(|\Gamma|,S)$, we have
\[
\varphi(\widehat m_{u_\alpha}^\pi(x))
=
\omega_{\varphi,x}^\pi(u_\alpha)
\longrightarrow
\omega_{\varphi,x}^\pi(1)
=\varphi(x).
\]
Hence the identity map belongs to the point-weak
closure of the linear subspace
\[
\{\widehat m_u^\pi:
u\in A_p(G),\ u\text{ finitely supported}\}
\]
of $B(B_u^p(|\Gamma|,S))$.

By the geometric form of the Hahn--Banach theorem,
the point-weak and point-norm closures of this
subspace coincide (cf. \cite[Page 133]{EKR} for an analogous situation in the $p=2$ case). Thus the identity map also
belongs to its point-norm closure, giving the
required net.
\end{proof}

Taking $\pi=\mathrm{id}_\Gamma$ shows that, if
$\Gamma$ has the $p$-AP, then the identity map on
$B_u^p(|\Gamma|,S)$ can be approximated in the
point-norm topology by Schur multipliers associated
to finitely supported elements of $A_p(\Gamma)$.

In the case $S=\mathbb{C}$, this shows that if $\Gamma$ has the $p$-AP, then $B_u^p(|\Gamma|)$ has a multiplier approximate identity with finite propagation, and this implies all ghost operators in $B_u^p(|\Gamma|)$ are compact 
\cite[Theorem 6.23 and Proposition 6.25]{DC}.

\begin{rem} \label{rem:pAP} \leavevmode
\begin{enumerate}
\item The $2$-AP implies exactness \cite[Theorem~2.1]{HK}.
\item For $1<p\leq q\leq 2$ or $2\leq q\leq p<\infty$, if $\Gamma$ has the $q$-AP, then $\Gamma$ has the $p$-AP \cite[Proposition 6.3]{ALR}.
\item For $p,q\in(1,\infty)$ with $1/p+1/q=1$, $\Gamma$ has the $p$-AP if and only if $\Gamma$ has the $q$-AP \cite[Proposition 2.3]{Ver}.
\item Amenable groups have the $p$-AP for all $p\in(1,\infty)$.
More generally, $p$-weakly amenable groups have the $p$-AP.
Every $2$-weakly amenable group is $p$-weakly amenable for all
$p\in(1,\infty)$ \cite[Proposition~6.2]{ALR}. Examples include
the free group $\mathbb F_2$, $\mathrm{SL}(2,\mathbb Z)$, and
$\mathrm{Sp}(1,n)$ \cite{ALR,CowH,CH}. We refer the reader to
\cite{Ver1} for a survey on $2$-weak amenability.
\item By (ii), the Haagerup--Kraus approximation
property implies the $p$-AP for every $p\in(1,\infty)$.
For example,
$\mathbb Z^2\rtimes\mathrm{SL}(2,\mathbb Z)$
has the $p$-AP for every such $p$
\cite{ALR,HK}. On the other hand,
$\mathrm{SL}(3,\mathbb Z)$ does not have the $p$-AP for any
$p\in(1,\infty)$ \cite[Theorem~1.5]{Ver}.
\end{enumerate}
\end{rem}

\subsection{Some approximation properties of $p$-operator spaces}

\begin{defn} \cite{ALR}
A $p$-completely bounded map $\theta:V\to W$ of $p$-operator spaces is said to have the $p$-operator space approximation property ($p$-OAP) if there is a net of bounded finite rank maps $T_\alpha:V\to W$ converging to $\theta$ in the stable point-norm topology, i.e., \[(id_{K(\ell^p)}\otimes T_\alpha)(a)\to (id_{K(\ell^p)}\otimes \theta)(a)\] in norm for all $a\in K(\ell^p)\stackrel{\vee_p}{\otimes} V$.

A $p$-operator space $V$ has the $p$-operator space approximation property if $id_V$ has the $p$-operator space approximation property.

The strong $p$-operator space approximation property is defined by replacing $K(\ell^p)$ by $B(\ell^p)$ above.
\end{defn} 

\begin{rem}
In this paper, we will only consider the $p$-OAP for $p$-operator spaces on $L^p$ space, so the $p$-operator space injective tensor product in the definition above can be replaced by the spatial tensor product.
\end{rem}



\begin{thm} \cite[Theorem 5.5]{ALR} \label{thm:ALRAP}
The following are equivalent for a discrete group $\Gamma$:
\begin{enumerate}
\item $\Gamma$ has the $p$-AP. 
\item $F^p_\lambda(\Gamma)$ has the strong $p$-OAP.
\item $F^p_\lambda(\Gamma)$ has the $p$-OAP.
\end{enumerate}
\end{thm}

\begin{defn} \cite{ALR}
A $p$-operator space $V$ has the $p$-completely bounded approximation property ($p$-CBAP) if there exists a constant $C\geq 1$ such that $id_V$ can be approximated in the point-norm topology by a net of $p$-completely bounded finite rank maps $T_\alpha$ with $\| T_\alpha \|_{pcb}\leq C$ for all $\alpha$.

We denote the infimum of all such constants $C$ by $\Lambda_{pcb}(V)$, with $\Lambda_{pcb}(V)=\infty$ if $V$ does not have the $p$-CBAP. 
When one can take $C=1$, we say that $V$ has the $p$-completely contractive approximation property ($p$-CCAP).
\end{defn}

The $p$-CBAP implies the strong $p$-OAP, since the uniform bound on the $p$-completely bounded norms allows pointwise norm convergence on the algebraic tensor product to extend to its completion.

By \cite[Theorem~5.4]{ALR}, for every discrete group $\Gamma$ and every $p\in(1,\infty)$, we have
\[
    \Lambda_p(\Gamma)
    =\Lambda_{pcb}(F_\lambda^p(\Gamma)).
\]
In particular, $\Gamma$ is $p$-weakly amenable if and only if $F_\lambda^p(\Gamma)$ has the $p$-CBAP.

\begin{lem} \label{lem:controlled-lifting}
Let $p\in[1,\infty)$, let $V$ and $W$ be $p$-operator spaces, and
let $\theta:V\longrightarrow W$ be a $p$-complete contraction. Suppose that, for every
$n\in\mathbb N$, every nonzero $y\in M_n(W)$, every $\eta>0$, and
every $r\in(0,1)$, there exists $x\in M_n(V)$ such that
\[
\|x\|<(1+\eta)\|y\|
\]
and
\[
\|\theta_n(x)-y\|<r\|y\|.
\]
Then $\theta$ is a $p$-complete quotient map.
\end{lem}

\begin{proof}
Fix $n\in\mathbb N$, a nonzero element
$y\in M_n(W)$, and $\varepsilon>0$.
The case $y=0$ is immediate.
Choose $\eta>0$ and $r\in(0,1)$ such that
\[
\frac{1+\eta}{1-r}<1+\varepsilon.
\]
Put $y_0=y$. Inductively, if $y_{j-1}\neq 0$,
choose $x_j\in M_n(V)$ such that
\[
\|x_j\|<(1+\eta)\|y_{j-1}\|
\]
and, putting
\[
y_j=y_{j-1}-\theta_n(x_j),
\]
we have $\|y_j\|<r\|y_{j-1}\|$.
If $y_{j-1}=0$, take $x_j=0$ and $y_j=0$.
In either case,
\[
\|x_j\|\leq(1+\eta)\|y_{j-1}\|,
\qquad
\|y_j\|\leq r\|y_{j-1}\|.
\]
Then
\[
\|y_j\|\leq r^j\|y\|
\]
and
\[
\sum_{j=1}^\infty\|x_j\|
\leq
(1+\eta)\sum_{j=0}^\infty r^j\|y\|
<
(1+\varepsilon)\|y\|.
\]
Hence the series
\[
x=\sum_{j=1}^\infty x_j
\]
converges in $M_n(V)$ and satisfies
\[
\theta_n(x)=y,
\qquad
\|x\|<(1+\varepsilon)\|y\|.
\]
Since $\theta_n$ is contractive, the quotient norm induced by
$\theta_n$ is therefore equal to the given norm on $M_n(W)$.
As this holds for every $n$, $\theta$ is a $p$-complete quotient
map.
\end{proof}

\section{Property A and approximation properties of $\ell^p$ uniform Roe algebras}
\label{sec:property-A}

In this section, we study approximation properties of
$\ell^p$ uniform Roe algebras and the role of property A.
We establish $p$-nuclearity, a coefficient-entry
characterization, and exactness of the coefficient
functor under suitable hypotheses. We also show that
the corresponding results at the endpoint $p=1$
hold without assuming property A.

For $p\in[1,\infty)$ and $n\in\mathbb{N}$, we write $M_n^p$ for $M_n(\mathbb{C})$ regarded as the algebra of bounded linear operators on $n$-dimensional $\ell^p$ space.

\begin{defn} \cite[Proposition 5.1]{ALR} \label{def:pnuc}
For $p\in[1,\infty)$, an $L^p$-operator algebra $A$ is said to be $p$-nuclear if there exist nets of $p$-complete contractions $\phi_\alpha:A\to M_{n(\alpha)}^p$ and $\psi_\alpha:M_{n(\alpha)}^p\to A$ such that $\Vert \psi_\alpha\circ \phi_\alpha(a)-a\Vert\to 0$ for all $a\in A$.
\end{defn}

It follows immediately from the definition that every $p$-nuclear
$L^p$-operator algebra has the $p$-CCAP.

Equivalently, an $L^p$-operator algebra $A$ is $p$-nuclear if,
for every finite subset $\mathcal F\subseteq A$ and every
$\varepsilon>0$, there exist $p$-complete contractions
\[
\varphi:A\longrightarrow M_n^p,\qquad
\psi:M_n^p\longrightarrow A
\]
such that $\|\psi\circ\varphi(a)-a\|<\varepsilon$ for every
$a\in\mathcal F$.

The following lemma is a $p$-operator space analogue
of the bounded matrix product argument used in
\cite[Theorem~4.3.9]{WillettPA}.

\begin{lem}
\label{lem:bounded-matrix-products}
Let $p\in[1,\infty)$, and let $(F_i)_{i\in I}$ be a family
of nonempty finite sets with $\sup_{i\in I}|F_i|<\infty$.
Then
\[
A=\prod_{i\in I}B(\ell^p(F_i)),
\]
where the product is taken with the supremum norm
and is represented diagonally on
$\bigoplus_{i\in I}^p\ell^p(F_i)$, is $p$-nuclear.
\end{lem}

\begin{proof}
For each $i\in I$, identify $\ell^p(F_i)$ with
$\ell^p_{|F_i|}$ by choosing an enumeration of $F_i$.
Fix a finite subset $\mathcal F=\{a^1,\ldots,a^m\}$ of $A$
and $\varepsilon>0$.

Write $a^j=(a_i^j)_{i\in I}$ for $1\leq j\leq m$.
For each $k$ occurring among the cardinalities of the $F_i$,
consider the set
\[
\big\{(a_i^1,\ldots,a_i^m): |F_i|=k\big\} \subseteq (M_k^p)^m,
\]
where $(M_k^p)^m$ is equipped with the maximum norm.
This set is bounded since
\[
\max_{1\leq j\leq m}\|a_i^j\|
\leq\max_{1\leq j\leq m}\|a^j\|.
\]
By finite-dimensionality, it is therefore totally bounded.
Hence there is a finite subset $J\subseteq I$ and a map
$\sigma:I\to J$ such that $|F_i|=|F_{\sigma(i)}|$ and
\[
\max_{1\leq j\leq m}
\|a_i^j-a_{\sigma(i)}^j\|<\frac{\varepsilon}{2}
\]
for all $i\in I$.

Put $n=\sum_{i\in J}|F_i|$.
Define $\varphi:A\to M_n^p$ by restriction to the
coordinates in $J$, followed by the diagonal embedding.
Define $\psi:M_n^p\to A$ by
\[
\psi(b)_i=b_{\sigma(i),\sigma(i)},
\]
where $b_{j,j}$ denotes the diagonal block of $b$
corresponding to $F_j$.

Both maps are $p$-complete contractions, since coordinate
restriction, diagonal embedding, and diagonal block
compression are contractive at every matrix level.
Moreover, for $1\leq j\leq m$,
\[
\|\psi\circ\varphi(a^j)-a^j\|
=\sup_{i\in I}\|a_{\sigma(i)}^j-a_i^j\|
\leq\frac{\varepsilon}{2}<\varepsilon.
\]
Thus $A$ is $p$-nuclear.
\end{proof}

Now we turn our attention to property A, a coarse geometric property introduced by Yu \cite{Yu}, who was motivated by its role as a sufficient condition for coarse embeddability into Hilbert space, and thus for the strong Novikov conjecture.
Many equivalent definitions for this property have been found, and we shall use the following due to Higson-Roe.

\begin{defprop} \cite{HR}
A discrete metric space $X$ with bounded geometry has property A if and only if for every $p\in[1,\infty)$ and all $R,\varepsilon>0$, there exists a family $(\xi_x)_{x\in X}$ of vectors in $\ell^p(X)$ such that
\begin{enumerate}
\item $\|\xi_x\|_p=1$ for all $x\in X$,
\item there exists $S>0$ such that $\xi_x$ is supported in $B(x,S)$ for each $x\in X$,
\item if $d(x,y)\leq R$, then $\|\xi_x-\xi_y\|_p<\varepsilon$.
\end{enumerate}
We may also assume that $\xi_x(y)\geq 0$ for all $x,y\in X$.
\end{defprop}

\begin{lem}
\label{lem:local-factorization}
Let $p\in[1,\infty)$, let $q$ be the conjugate exponent,
and let $X$ be a metric space with bounded geometry.
Suppose that $(F_i)_{i\in I}$ is a family of nonempty finite
subsets of $X$ of uniformly bounded diameter, and that
$(v_i)_{i\in I}$ and $(u_i)_{i\in I}$ are families of
nonnegative functions on $X$, supported in the respective
sets $F_i$, such that
\[
\|(v_i(x))_{i\in I}\|_{\ell^p(I)}\leq 1,\qquad
\|(u_i(x))_{i\in I}\|_{\ell^q(I)}\leq 1
\]
for every $x\in X$.

For every $SQ_p$ space $E$ and every closed $p$-operator space
$S\subseteq B(E)$, there are $p$-complete contractions
\[
B_u^p(X,S)
\xrightarrow{\ \eta_S\ }
\prod_{i\in I}M_{|F_i|}(S)
\xrightarrow{\ \psi_S\ }
B_u^p(X,S)
\]
such that $\psi_S\circ\eta_S$ is Schur multiplication
by the kernel
\[
k(x,y)=\sum_{i\in I}u_i(x)v_i(y).
\]
Moreover, every operator in the range of $\psi_S$ has
propagation at most $\sup_{i\in I}\operatorname{diam}(F_i)$.
\end{lem}

\begin{proof}
Let $P_i$ and $\iota_i$ denote the canonical projection
onto and inclusion of $\ell^p(F_i,E)$, respectively.
Define
\[
\eta_S(a)=(P_i a\iota_i)_{i\in I}.
\]
Since coordinate compression is contractive at every
matrix level, $\eta_S$ is a $p$-complete contraction.

Put
\[
\mathcal E=\bigoplus_{i\in I}^p\ell^p(F_i,E),
\]
and define $V:\ell^p(X,E)\to\mathcal E$ and
$U:\mathcal E\to\ell^p(X,E)$ by
\[
(V\zeta)_i(x)=v_i(x)\zeta(x),\qquad
(U(\zeta_i)_{i\in I})(x)=\sum_{i\in I}u_i(x)\zeta_i(x),
\]
where each $\zeta_i$ is extended by zero outside $F_i$.
The hypotheses and H\"older's inequality imply that both
$U$ and $V$ are contractions.

For $b=(b_i)_{i\in I}\in\prod_{i\in I}M_{|F_i|}(S)$,
define
\[
\psi_S(b)=U\Big(\bigoplus_{i\in I}b_i\Big)V.
\]
Then
\[
\|\psi_S(b)\|\leq\sup_{i\in I}\|b_i\|.
\]
Using the canonical isometric identifications
\[
\ell_n^p(\ell^p(X,E))
\cong \ell^p(X,\ell_n^p(E))
\]
and
\[
\ell_n^p(\mathcal E)
\cong
\bigoplus_{i\in I}^p\ell^p(F_i,\ell_n^p(E)),
\]
the same argument with $E$ replaced by $\ell_n^p(E)$
shows that $\psi_S$ is contractive at every matrix
level, and hence is a $p$-complete contraction.

The $(x,y)$-entry of $\psi_S(b)$ is
\[
\sum_{\substack{i\in I\\x,y\in F_i}}
u_i(x)v_i(y)(b_i)_{x,y}.
\]
This series converges absolutely in operator norm by
H\"older's inequality, and hence its sum belongs to $S$.
Furthermore, the entry is zero unless $x$ and $y$ belong
to a common $F_i$. Thus $\psi_S(b)$ belongs to
$\mathbb C[X,S]$ and has the asserted propagation bound.

Taking $b=\eta_S(a)$ in the preceding formula yields
\[
(\psi_S\circ\eta_S)(a)_{x,y}=k(x,y)a_{x,y},
\]
as required.
\end{proof}

For a scalar-valued kernel $k$ on $X\times X$,
we write $M_k$ for the Schur multiplication map
\[
M_k([a_{x,y}])=[k(x,y)a_{x,y}]
\]
whenever this defines a bounded map on the
operator space under consideration.

\begin{lem} (cf.\ \cite[Lemma~6.17]{DC})
\label{lem:local-schur-approximation}
Let $X$ be a metric space with bounded geometry and let
$p\in[1,\infty)$. Suppose that either $p=1$ or $X$ has
property A.

For every $SQ_p$ space $E$, every closed $p$-operator space
$S\subseteq B(E)$, every finite subset
$\mathcal F\subseteq B_u^p(X,S)$, and every
$\varepsilon>0$, there exist $L\geq 0$ and families
$(u_i)_{i\in X}$ and $(v_i)_{i\in X}$ as in
Lemma~\ref{lem:local-factorization}, with $F_i=B(i,L)$,
such that the corresponding kernel $k$ satisfies
\[
\|M_k(a)-a\|<\varepsilon
\]
for all $a\in\mathcal F$.

If $p=1$, we may take
\[
k(x,y)=\mathbf 1_{\{d(x,y)\leq L\}}.
\]
If $1<p<\infty$ and $X$ has property A, writing
$q$ for the conjugate exponent of $p$, we may take
\[
k(x,y)=\langle\xi_x^{p/q},\xi_y\rangle,
\]
where $(\xi_x)_{x\in X}$ is a family of nonnegative unit
vectors in $\ell^p(X)$ with
$\operatorname{supp}(\xi_x)\subseteq B(x,L)$.

In both cases $M_k$ is $p$-completely contractive.
\end{lem}

\begin{proof}
By Lemma~\ref{lem:local-factorization}, it suffices to
consider the case where all elements of $\mathcal F$
have finite propagation. Choose $R\geq 0$ such that
$\operatorname{prop}(a)\leq R$ for every $a\in\mathcal F$,
and put
\[
M=\max\big(\{1\}\cup\{\|a\|:a\in\mathcal F\}\big).
\]
Let $N_R$ be a uniform bound on the cardinalities of
balls of radius $R$.

Suppose first that $p=1$. Take $L=R$ and put
\[
v_i(y)=\mathbf 1_{\{i\}}(y),\qquad
u_i(x)=\mathbf 1_{B(i,R)}(x).
\]
These families satisfy the hypotheses of
Lemma~\ref{lem:local-factorization}, and the corresponding
kernel is precisely
$\mathbf 1_{\{d(x,y)\leq R\}}$.
In particular, $M_k(a)=a$ for every $a\in\mathcal F$.

Now suppose that $1<p<\infty$.
By property A, there exist $L>0$ and nonnegative unit
vectors $\xi_x\in\ell^p(X)$ supported in $B(x,L)$ such
that
\[
\|\xi_x-\xi_y\|_p<\frac{\varepsilon}{2MN_R}
\]
whenever $d(x,y)\leq R$.
Put
\[
v_i(y)=\xi_y(i),\qquad
u_i(x)=\xi_x(i)^{p/q}.
\]
These families satisfy the hypotheses of
Lemma~\ref{lem:local-factorization}, and their corresponding
kernel is $k(x,y)=\langle\xi_x^{p/q},\xi_y\rangle$.
Moreover,
\[
|1-k(x,y)|
\leq \|\xi_x-\xi_y\|_p
<\frac{\varepsilon}{2MN_R}
\]
whenever $d(x,y)\leq R$.

Consequently, the standard estimate for operators of
propagation at most $R$ gives
\[
\|M_k(a)-a\|
\leq N_R\|a\|
\sup_{d(x,y)\leq R}|1-k(x,y)|
\leq\frac{\varepsilon}{2}<\varepsilon
\]
for every $a\in\mathcal F$.
The conclusion for arbitrary finite subsets follows
by density and the contractivity of $M_k$.
\end{proof}

For $1<p<\infty$, property A is in fact equivalent to the
existence, for every $R,\delta>0$, of families of vectors
as above for which $k(x,y)>1-\delta$ whenever $d(x,y)\leq R$.
Indeed, the forward implication follows from the preceding
construction. For the converse, observe that
\[
\left\|\frac{\xi_x+\xi_y}{2}\right\|_p
\geq
\left\langle\xi_x^{p/q},
\frac{\xi_x+\xi_y}{2}\right\rangle
>1-\frac{\delta}{2}
\]
whenever $d(x,y)\leq R$.
The uniform convexity of $\ell^p(X)$ then implies that
$\|\xi_x-\xi_y\|_p$ can be made arbitrarily small on such
pairs by choosing $\delta$ sufficiently small.

\begin{thm} \label{Apnuc}
Let $X$ be a discrete metric space with bounded geometry and let $p\in[1,\infty)$.
If either $p=1$ or $X$ has property A, then $B_u^p(X)$ is $p$-nuclear.
\end{thm}

\begin{proof}
Fix a finite subset $\mathcal F\subseteq B_u^p(X)$
and $\varepsilon>0$.
By Lemmas~\ref{lem:local-factorization}
and~\ref{lem:local-schur-approximation},
there are $p$-complete contractions
\[
B_u^p(X)
\xrightarrow{\ \eta\ }
C
\xrightarrow{\ \psi\ }
B_u^p(X),
\qquad
C=\prod_{i\in X}B(\ell^p(B(i,L))),
\]
for some $L\geq 0$, such that
\[
\|\psi\circ\eta(a)-a\|<\frac{\varepsilon}{2}
\]
for every $a\in\mathcal F$.

Since $X$ has bounded geometry,
Lemma~\ref{lem:bounded-matrix-products} implies that
$C$ is $p$-nuclear.
Hence there are $p$-complete contractions
$\alpha:C\to M_n^p$ and $\beta:M_n^p\to C$ such that
\[
\|\beta\circ\alpha(\eta(a))-\eta(a)\|
<\frac{\varepsilon}{2}
\]
for every $a\in\mathcal F$.
It follows that $\alpha\circ\eta$ and $\psi\circ\beta$
are $p$-complete contractions whose composition
approximates every element of $\mathcal F$ within
$\varepsilon$.
Thus $B_u^p(X)$ is $p$-nuclear.
\end{proof}

At the endpoint $p=1$, the local factorization in
Lemma~\ref{lem:local-schur-approximation} does not require
property A. In this case, Schur multiplication by the
characteristic function of a controlled neighbourhood of
the diagonal is $1$-completely contractive.

\begin{thm}
\label{thm:roe-coefficients}
Let $X$ be a discrete metric space with bounded geometry
and let $p\in[1,\infty)$. Suppose that either $p=1$
or $X$ has property A. Then the following statements hold.

\begin{enumerate}
\item[(i)]
For every $SQ_p$ space $E$ and every closed
$p$-operator space $S\subseteq B(E)$,
\[
B_u^p(X,S)
=
\{a\in B_u^p(X,B(E)):a_{x,y}\in S
\text{ for all }x,y\in X\}.
\]

\item[(ii)]
For every short exact sequence
\[
0\longrightarrow J\longrightarrow A
\xrightarrow{\pi}B\longrightarrow 0
\]
of concrete $p$-operator spaces on $SQ_p$ spaces,
the induced sequence
\[
0\longrightarrow B_u^p(X,J)
\longrightarrow B_u^p(X,A)
\xrightarrow{\widehat{\pi}}B_u^p(X,B)
\longrightarrow 0
\]
is exact, and $\widehat{\pi}$ is a $p$-complete quotient
map. Equivalently, the canonical map
\[
B_u^p(X,A)/B_u^p(X,J)
\longrightarrow B_u^p(X,B)
\]
is a $p$-complete isometry.
\end{enumerate}
\end{thm}

\begin{proof}
(i) The inclusion from left to right is immediate.
Conversely, suppose that $a\in B_u^p(X,B(E))$ and that
all matrix entries of $a$ belong to $S$.
By Lemma~\ref{lem:local-schur-approximation}, $a$ can be
approximated in norm by operators $M_k(a)$ of finite
propagation whose matrix entries belong to $S$.
Thus $a\in B_u^p(X,S)$.

(ii) Entrywise application of $\pi$ defines a
$p$-complete contraction
\[
\widehat{\pi}:B_u^p(X,A)\longrightarrow B_u^p(X,B).
\]
We verify the hypothesis of Lemma~\ref{lem:controlled-lifting}.

Fix $n\in\mathbb N$, a nonzero element
$b\in M_n(B_u^p(X,B))$, $\eta>0$, and $r\in(0,1)$.
Using the canonical permutation of coordinates,
identify
\[
M_n(B_u^p(X,B))=B_u^p(X,M_n(B)),
\]
and similarly for $A$.

By Lemma~\ref{lem:local-schur-approximation}, choose
$L\geq 0$ and families $(u_i)_{i\in X}$ and
$(v_i)_{i\in X}$ as in
Lemma~\ref{lem:local-factorization}, with $F_i=B(i,L)$,
such that the corresponding kernel
\[
k(x,y)=\sum_{i\in X}u_i(x)v_i(y)
\]
satisfies
\[
\|M_k(b)-b\|<r\|b\|.
\]
Write
\[
(b_i)_{i\in X}=\eta_{M_n(B)}(b).
\]
Then $\|b_i\|\leq\|b\|$ for every $i\in X$.

Since $\pi$ is a $p$-complete quotient map, for each
$i\in X$ we may choose
$c_i\in M_{|F_i|}(M_n(A))$ such that
\[
(\pi_n)_{|F_i|}(c_i)=b_i,\qquad
\|c_i\|\leq(1+\eta/2)\|b_i\|.
\]
For $b_i=0$, we take $c_i=0$.

Let $\psi_{M_n(A)}$ be the reconstruction map in
Lemma~\ref{lem:local-factorization} determined by the
same families $(u_i)_{i\in X}$ and $(v_i)_{i\in X}$
used to define $k$. Put
\[
a=\psi_{M_n(A)}((c_i)_{i\in X}).
\]
Then $a\in M_n(B_u^p(X,A))$ and
\[
\|a\|\leq\sup_{i\in X}\|c_i\|
\leq(1+\eta/2)\|b\|<(1+\eta)\|b\|.
\]
By the matrix-entry formula in
Lemma~\ref{lem:local-factorization},
\[
(\widehat{\pi})_n(a)
=\psi_{M_n(B)}((b_i)_{i\in X})
=M_k(b).
\]
Consequently,
\[
\|(\widehat{\pi})_n(a)-b\|<r\|b\|.
\]
Lemma~\ref{lem:controlled-lifting} now implies that
$\widehat{\pi}$ is a $p$-complete quotient map.

Finally, $B_u^p(X,J)\subseteq\ker\widehat{\pi}$.
Conversely, if $a\in\ker\widehat{\pi}$, then every
matrix entry of $a$ belongs to $J$.
By (i), $a\in B_u^p(X,J)$.
Hence $\ker\widehat{\pi}=B_u^p(X,J)$, which completes
the proof.
\end{proof}

Thus, at the endpoint $p=1$, neither the coefficient-entry
characterization nor exactness of the coefficient functor
detects property A. Since every Banach space belongs to
$SQ_1$, there is no restriction on the choice of the
ambient Banach space $E$ in the coefficient-entry
characterization at this endpoint.

\begin{cor}\label{cor:controlled-CCAP}
Let $X$ be a discrete metric space with bounded geometry
and let $p\in[1,\infty)$. Suppose that either $p=1$
or $X$ has property A.

Then there is a net $(\Phi_\alpha)$ of finite-rank
$p$-complete contractions on $B_u^p(X)$ such that
\[
\Phi_\alpha(a)\longrightarrow a
\]
in norm for every $a\in B_u^p(X)$, and, for each
$\alpha$, there exists $R_\alpha<\infty$ such that
\[
\operatorname{prop}(\Phi_\alpha(a))
\leq R_\alpha
\]
for every $a\in B_u^p(X)$.
\end{cor}

\begin{proof}
The approximating maps constructed in the proof
of Theorem~\ref{Apnuc} have the form
\[
\Phi=\psi\circ\beta\circ\alpha\circ\eta,
\]
where $\alpha$ takes values in a finite-dimensional
matrix algebra. Hence $\Phi$ is a finite-rank
$p$-complete contraction.

By Lemma~\ref{lem:local-factorization}, every operator
in the range of $\psi$ has propagation at most $2L$,
where $L$ is the radius used in the corresponding
local factorization. Directing these approximating maps
by finite subsets of $B_u^p(X)$ and positive tolerances
gives the required net.
\end{proof}

By \cite{Ozawa,GK}, a countable discrete group
$\Gamma$ is exact if and only if $|\Gamma|$ has
property A.

\begin{thm} (cf. \cite[Theorem 2.3]{Zach}) \label{thm:functor}
Let $\Gamma$ be a countable discrete group.
The following conditions are equivalent.
\begin{enumerate}
\item[(i)]
$\Gamma$ is exact.

\item[(ii)]
For every $p\in(1,\infty)$, every $SQ_p$ space $E$,
and every closed $p$-operator space $S\subseteq B(E)$,
\[
B_u^p(|\Gamma|,S)
=
\{a\in B_u^p(|\Gamma|,B(E)):
a_{s,t}\in S\text{ for all }s,t\in\Gamma\}.
\]

\item[(iii)]
For every $p\in(1,\infty)$ and every short exact
sequence
\[
0\longrightarrow J\longrightarrow A
\longrightarrow B\longrightarrow 0
\]
of concrete $p$-operator spaces on $SQ_p$ spaces,
the induced sequence
\[
0\longrightarrow B_u^p(|\Gamma|,J)
\longrightarrow B_u^p(|\Gamma|,A)
\longrightarrow B_u^p(|\Gamma|,B)
\longrightarrow 0
\]
is exact, and the right-hand map is a $p$-complete
quotient map.
\end{enumerate}
\end{thm}

\begin{proof}
If $\Gamma$ is exact, then $|\Gamma|$ has property A,
so (ii) and (iii) follow from
Theorem~\ref{thm:roe-coefficients}.

Conversely, taking $p=2$ in (ii) gives condition (3)
of \cite[Theorem~2.3]{Zach}.
Taking $p=2$ in (iii) and considering short exact
sequences of $C^*$-algebras gives condition (4) of
the same theorem.
In either case, $\Gamma$ is exact.
\end{proof}

\section{The $p$-invariant translation approximation property ($p$-ITAP) for discrete groups}

In this section, we introduce and study the $p$-invariant translation approximation property for discrete groups, primarily for $p\in(1,\infty)$. We also record the endpoint $p=1$, where the full coefficient
fixed-point equality turns out to hold for every countable discrete group.

Let $\Gamma$ be a countable discrete group and let
$|\Gamma|$ be its metric realization as above. The algebras considered below do
not depend on the choice of such a metric. Indeed, an operator
$T=[T_{s,t}]_{s,t\in\Gamma}$ has finite propagation if and only if
the set $\{st^{-1}: T_{s,t}\neq 0\}$ is finite. Consequently, the algebraic uniform Roe algebra
$\mathbb{C}_u^p[|\Gamma|]$, and more generally $\mathbb{C}[|\Gamma|,S]$ for any
coefficient space $S$, is independent of the choice of proper
right-invariant metric, as are their completions and the
corresponding invariant translation approximation properties.

The finite-propagation operators whose matrix entries satisfy $T_{sr,tr}=T_{s,t}$ for all $s,t,r\in\Gamma$ are precisely the elements of $\lambda_p(\mathbb C\Gamma)$.
In particular, $F_\lambda^p(\Gamma)$ is naturally contained in $B_u^p(|\Gamma|)$.

The adjoint action $\mathrm{Ad}(\rho)$ of $\Gamma$ on $B(\ell^p(\Gamma))$ is given by conjugation with $\rho$, i.e., $\mathrm{Ad}(\rho_g)(T) = \rho_gT\rho_{g^{-1}}$ for $g\in\Gamma$ and $T\in B(\ell^p(\Gamma))$.
Note that
\[
(\mathrm{Ad}(\rho_g)(T))_{s,t}
=
\delta_{sg}^*(T\delta_{tg})
=
T_{sg,tg}
\]
for $s,t,g\in\Gamma$.
Since $\mathrm{Ad}(\rho_g)$ sends finite propagation operators to finite propagation operators, the adjoint action preserves the $\ell^p$ uniform Roe algebra.
Moreover, as discussed in the previous paragraph, the group ring consists precisely of the elements in $\mathbb{C}^p_u[|\Gamma|]$ fixed by all $\mathrm{Ad}(\rho_g)$.

For an $SQ_p$ space $E$, the corresponding action on $B(\ell^p(\Gamma,E))$ is given by
\[
g\cdot T
=
(\rho_g\otimes I_E)T(\rho_{g^{-1}}\otimes I_E).
\]
This action preserves $B_u^p(|\Gamma|,S)$ for every $p$-operator space $S\subseteq B(E)$. An operator $T$ is fixed by this action if and only if
\[
T_{sg,tg}=T_{s,t}
\]
for all $s,t,g\in\Gamma$.

Writing $B_u^p(|\Gamma|)^\Gamma$ for the fixed-point algebra of the adjoint action, we have
\[
B_u^p(|\Gamma|)^\Gamma
=
B_u^p(|\Gamma|)\cap CV_p(\Gamma).
\]
It was shown by Daws and Spronk \cite[Theorem 1.1]{DS} that if $\Gamma$ has the approximation property of Haagerup-Kraus \cite{HK}, then $CV_p(\Gamma)$ coincides with the algebra of $p$-pseudomeasures $PM_p(\Gamma)$. Vergara generalized this result, showing that if $\Gamma$ has the $p$-approximation property, then $PM_p(\Gamma)=CV_p(\Gamma)$ \cite[Theorem 6.1]{Ver}.
In general, one only has the inclusion $PM_p(\Gamma)\subseteq CV_p(\Gamma)$, and the bicommutant $PM_p(\Gamma)''$ equals $CV_p(\Gamma)$ \cite[Theorem 1.2]{DS}.

We clearly have $F^p_\lambda(\Gamma)\subseteq B_u^p(|\Gamma|)^\Gamma$ for every countable discrete group $\Gamma$.

\begin{defn} \label{def:pITAP}
For $1<p<\infty$, a countable discrete group $\Gamma$ is said to have the $p$-invariant translation approximation property ($p$-ITAP) if \[ F^p_\lambda(\Gamma) = B_u^p(|\Gamma|)^\Gamma. \]
\end{defn}

The same fixed-point equality also makes sense at the endpoint $p=1$.
We shall say that a countable discrete group has the \emph{$1$-ITAP} if
\[
F_\lambda^1(\Gamma)=B_u^1(|\Gamma|)^\Gamma.
\]

\begin{prop}\label{prop:1ITAP}
Let $\Gamma$ be a countable discrete group, let $E$ be an $SQ_1$ space, and
let $S\subseteq B(E)$ be a closed $1$-operator space. Then
\[
B_u^1(|\Gamma|,S)^\Gamma
=
F_\lambda^1(\Gamma)\otimes S.
\]
In particular, every countable discrete group has the $1$-operator ITAP.
\end{prop}

\begin{proof}
Let $T\in B_u^1(|\Gamma|,S)^\Gamma$.
For $R\geq 0$, let $M_R$ denote Schur multiplication
by the kernel
\[
(s,t)\longmapsto \mathbf 1_{\{d(s,t)\leq R\}}.
\]
By Lemma~\ref{lem:local-schur-approximation}, these maps
are $1$-complete contractions. Since $M_R$ fixes every
operator of propagation at most $R$, it follows by
density that $M_R(T)\to T$ in norm as $R\to\infty$.

By right invariance of the metric and $\Gamma$-invariance
of $T$, we have
\[
M_R(T)=\sum_{g\in B(e,R)}\lambda_g\otimes T_{g,e}
\in F_\lambda^1(\Gamma)\otimes S.
\]
Taking the norm limit gives the desired inclusion.
The reverse inclusion is immediate.
\end{proof}

Since every Banach space is an $SQ_1$ space, there is no restriction
on the choice of the ambient Banach space $E$ in the preceding
coefficient fixed-point identity.

The 2-ITAP was introduced by Roe \cite[Section 11.5.3]{RoeL}.
For conjugate exponents $p,q\in(1,\infty)$, taking
transposes gives an isometric anti-isomorphism
from $B(\ell^p(\Gamma))$ onto $B(\ell^q(\Gamma))$.
Since transposition preserves finite propagation
and sends $\lambda_g^p$ to $\lambda_{g^{-1}}^q$,
it restricts to isometric anti-isomorphisms
\[
F_\lambda^p(\Gamma)\longrightarrow F_\lambda^q(\Gamma)
\]
and
\[
B_u^p(|\Gamma|)^\Gamma
\longrightarrow B_u^q(|\Gamma|)^\Gamma.
\]
Consequently, $\Gamma$ has the $p$-ITAP if and only
if it has the $q$-ITAP.

In this paper, we shall study the following stronger property by introducing coefficients in $p$-operator spaces in the spirit of Zacharias in the $p=2$ case \cite{Zach}.
Note that in general, we have the inclusions \[ F^p_\lambda(\Gamma)\otimes S \subseteq B_u^p(|\Gamma|)^\Gamma\otimes S \subseteq (B_u^p(|\Gamma|)\otimes S)^\Gamma \subseteq B^p_u(|\Gamma|,S)^\Gamma \]  for any countable discrete group $\Gamma$ and any $p$-operator space $S$, where $\Gamma$ acts trivially on $S$.

\begin{defn} \label{def:coefficient-ITAP}
Let $p\in[1,\infty)$ and let $\Gamma$ be a countable discrete group.
Let $S\subseteq B(E)$ be a $p$-operator space, where $E$ is an $SQ_p$ space.
We say that $\Gamma$ has the $p$-invariant translation approximation property for $S$ if
\[ F^p_\lambda(\Gamma)\otimes S = B^p_u(|\Gamma|,S)^\Gamma, \]
where $\Gamma$ acts trivially on $S$. 

If $\Gamma$ has the $p$-invariant translation approximation property for every $p$-operator space $S\subseteq K(\ell^p)$, then we say that $\Gamma$ has the $p$-operator invariant translation approximation property ($p$-operator ITAP).
\end{defn}

Given a discrete group $\Gamma$, a function $\phi:\Gamma\to\mathbb{C}$ is said to be of positive type if for all $n\in\mathbb{N}$, $g_1,\ldots,g_n\in\Gamma$, and $\lambda_1,\ldots,\lambda_n\in\mathbb{C}$, we have
\[ \sum_{i,j=1}^n \overline{\lambda_i}\lambda_j\phi(g_i^{-1}g_j)\geq 0. \]
A real-valued symmetric function $l$ on $\Gamma$ with $l(e)=0$ is said to be of negative type if
\[
\sum_{i,j=1}^n
\overline{\lambda_i}\lambda_j
l(g_i^{-1}g_j)\leq 0
\]
whenever $g_1,\ldots,g_n\in\Gamma$ and $\lambda_1,\ldots,\lambda_n\in\mathbb C$ satisfy $\sum_{i=1}^n\lambda_i=0$.
The function $\phi$ is said to be normalized if $\phi(e)=1$, where $e$ is the identity element in $\Gamma$.

Given a normalized function $\phi$ of positive type
on a discrete group $\Gamma$, define a kernel
$k:\Gamma\times\Gamma\to\mathbb C$ by
\[
k(s,t)=\phi(st^{-1}).
\]
Then $k$ is of positive type and is invariant under
the right diagonal action of $\Gamma$.

There exist a Hilbert space $\mathcal H$ and a map
$f:\Gamma\to\mathcal H$ whose range consists of
unit vectors such that
\[
k(s,t)=\langle f(s),f(t)\rangle
\]
\cite[Proposition~8.5(i)]{BL}.

Since Hilbert spaces are $SQ_p$ spaces for every
$p\in(1,\infty)$ \cite[Theorem~1]{Herz}, the
factorization characterization in
\cite[Theorem~8.3]{Daws10} gives
$\phi\in M_{cb}A_p(\Gamma)$ with
\[
\|\phi\|_{M_{cb}A_p(\Gamma)}=1.
\]
By Lemma~\ref{lem:vector-valued-Schur},
Schur multiplication by $k$ defines a $p$-complete
contraction on $B_u^p(|\Gamma|,S)$ whenever
$\Gamma$ is countable and $S$ is a concrete
$p$-operator space.

Moreover, by \cite[Proposition~8.4 and
Theorem~8.6]{Daws10}, the scalar Schur multiplier
$M_k$ is a $p$-complete contraction on
$B(\ell^p(\Gamma))$.
The right diagonal invariance of $k$ implies that
$M_k$ preserves $CV_p(\Gamma)$.

\begin{prop} (cf. \cite[Proposition 11.46]{RoeL})
\label{prop:positive-type-criterion}
Let $\Gamma$ be a countable discrete group. Suppose there is a net
$(\varphi_i)$ of normalized functions of positive type on $\Gamma$
such that
\[
\varphi_i(g)\longrightarrow 1
\]
for every $g\in\Gamma$. For each $i$, let
\[
k_i(g,h)=\varphi_i(gh^{-1}).
\]
If Schur multiplication $M_{k_i}$ maps $CV_p(\Gamma)$ into
$F_\lambda^p(\Gamma)$ for every $i$, then $\Gamma$ has the
$p$-ITAP.
\end{prop}

\begin{proof}
By the preceding discussion, each $M_{k_i}$ is a
$p$-complete contraction.

For $T\in \mathbb{C}_u^p[|\Gamma|]$, put
\[
D=\{st^{-1}:T_{s,t}\neq 0\}.
\]
Then $D$ is finite, and
\[
M_{k_i}(T)-T
=
\sum_{g\in D}(\varphi_i(g)-1)M_g(T),
\]
where $M_g$ denotes Schur multiplication by the kernel
$(s,t)\mapsto\mathbf 1_{\{g\}}(st^{-1})$.
Since $M_g(T)=\lambda_g D_g$, where $D_g$ is the
diagonal operator with diagonal entries $T_{gt,t}$,
we have $\|M_g(T)\|=\sup_{t\in\Gamma}|T_{gt,t}|\leq\|T\|$.
Consequently,
\[
\|M_{k_i}(T)-T\|
\leq\|T\|\sum_{g\in D}|\varphi_i(g)-1|
\longrightarrow 0.
\]
By contractivity and density, this convergence holds
for every $T\in B_u^p(|\Gamma|)$.

If $T\in B_u^p(|\Gamma|)^\Gamma$, then $T\in CV_p(\Gamma)$,
so $M_{k_i}(T)\in F_\lambda^p(\Gamma)$ by hypothesis.
Passing to the norm limit gives
$T\in F_\lambda^p(\Gamma)$.
\end{proof}

\begin{cor} \label{amenable}
Countable amenable discrete groups have the $p$-ITAP for all $p\in(1,\infty)$.
\end{cor}

\begin{proof}
Let $(F_i)$ be a F{\o}lner net for $\Gamma$, and define
\[
\varphi_i(s)
=
\frac{|F_i\cap sF_i|}{|F_i|},
\qquad s\in\Gamma.
\]
Then each $\varphi_i$ is finitely supported, normalized, and of
positive type, and
\[
\varphi_i(s)\longrightarrow1
\]
for every $s\in\Gamma$.

Let
\[
k_i(g,h)=\varphi_i(gh^{-1}).
\]
If $T\in CV_p(\Gamma)$ and
\[
b_g=T_{g,e},
\]
then
\[
M_{k_i}(T)
=
\sum_{g\in\operatorname{supp}(\varphi_i)}
\varphi_i(g)b_g\lambda_g
\in\mathbb C\Gamma
\subseteq F_\lambda^p(\Gamma).
\]
Hence $\Gamma$ has the $p$-ITAP by
Proposition~\ref{prop:positive-type-criterion}.
\end{proof}

Recall that a length function on a discrete group
$\Gamma$ is a function $l:\Gamma\to[0,\infty)$
satisfying $l(e)=0$, $l(g^{-1})=l(g)$, and
$l(gh)\leq l(g)+l(h)$ for all $g,h\in\Gamma$.

The following lemma uses a common argument for
polynomial and subexponential rapid decay estimates,
based on Schur multipliers arising from length
functions of negative type. For the classical
rapid decay setting, see also \cite{BN}.

\begin{lem}
\label{lem:rapid-decay-itap}
Let $\Gamma$ be a countable discrete group, let
$p\in(1,\infty)$, and let $l$ be a length function
of negative type on $\Gamma$. Suppose that for some
$\beta\in(0,1]$ there exists a function
$F:[0,\infty)\to[1,\infty)$ such that
\[
\log F(R)=o(R^\beta)
\quad\text{as }R\to\infty,
\]
and
\[
\|\lambda_p(f)\|
\leq F(R)\|f\|_{\ell^p(\Gamma)}
\]
whenever $f\in\mathbb C\Gamma$ is supported in
$B_l(e,R)=\{g\in\Gamma:l(g)\leq R\}$.
Then $\Gamma$ has the $p$-ITAP.
\end{lem}

\begin{proof}
Fix $t>0$. For $m\in\mathbb N\cup\{0\}$, put
\[
A_m=\{g\in\Gamma:m\leq l(g)<m+1\}.
\]
For every $f\in\mathbb C\Gamma$, we have
\begin{align*}
\|\lambda_p(f)\|
&\leq\sum_{m=0}^\infty
F(m+1)\|f\mathbf 1_{A_m}\|_{\ell^p(\Gamma)}\\
&\leq\sum_{m=0}^\infty
F(m+1)e^{-tm^\beta}
\|fe^{tl^\beta}\mathbf 1_{A_m}\|_{\ell^p(\Gamma)}\\
&\leq C_t\|fe^{tl^\beta}\|_{\ell^p(\Gamma)},
\end{align*}
where
\[
C_t=\sum_{m=0}^\infty F(m+1)e^{-tm^\beta}<\infty.
\]
Indeed, the assumption on $F$ implies that
$F(m+1)e^{-tm^\beta}\leq e^{-tm^\beta/2}$
for all sufficiently large $m$.

Since $l$ is of negative type, so is $l^\beta$
by \cite[Corollary~2.10]{BCR}.
Thus, by Schoenberg's theorem, the functions
\[
\phi_t(g)=e^{-tl(g)^\beta}
\]
are normalized functions of positive type.
Moreover, $\phi_t(g)\to 1$ for every $g\in\Gamma$
as $t\to 0$.

Let $T\in CV_p(\Gamma)$ and put
$b=T\delta_e\in\ell^p(\Gamma)$.
For each finite subset $K\subseteq\Gamma$, define
$f_K=\phi_t b\mathbf 1_K$.
The preceding estimate gives
\[
\|\lambda_p(f_K)-\lambda_p(f_{K'})\|
\leq C_t\|b(\mathbf 1_K-\mathbf 1_{K'})\|_{\ell^p(\Gamma)}.
\]
Hence, as $K$ increases over the finite subsets of
$\Gamma$, the operators $\lambda_p(f_K)$ converge in
norm to an element of $F_\lambda^p(\Gamma)$.

The $(s,r)$-entry of the limit is
$\phi_t(sr^{-1})b(sr^{-1})$.
Since $T_{s,r}=b(sr^{-1})$, the limit is precisely
$M_{k_t}(T)$, where
\[
k_t(s,r)=\phi_t(sr^{-1}).
\]
Therefore $M_{k_t}(CV_p(\Gamma))\subseteq
F_\lambda^p(\Gamma)$ for every $t>0$.
The conclusion follows from Proposition~\ref{prop:positive-type-criterion}.
\end{proof}

\begin{defn} \cite[Definition 3.5]{AOP}
For $p\in[1,\infty)$, a discrete group $\Gamma$ is said to have property RD$_p$ with respect to a length function $l$ if there exist $C>0$ and $s>0$ such that for all $f\in\mathbb{C}\Gamma$, we have
\[ \|f\|_{B(\ell^p(\Gamma))} \leq C\| f(1+l)^s \|_{\ell^p(\Gamma)}. \]
\end{defn}

This definition is equivalent to \cite[Definition 4.1]{LY} by \cite[Proposition 3.7]{AOP}.

\begin{prop} \label{prop:RDp}
For $p\in(1,\infty)$, if a countable discrete group $\Gamma$ has property RD$_p$ with respect to a length function of negative type, then $\Gamma$ has the $p$-ITAP.

In particular, if $\Gamma$ has property RD$_2$ with respect to a length function of negative type, then $\Gamma$ has the $p$-ITAP for all $p\in(1,\infty)$.
\end{prop}

\begin{proof}
Suppose that $\Gamma$ has property $\mathrm{RD}_p$
with respect to the length function $l$.
There exist constants $C,s>0$ such that
\[
\|\lambda_p(f)\|
\leq C(1+R)^s\|f\|_{\ell^p(\Gamma)}
\]
whenever $f\in\mathbb C\Gamma$ is supported in
$B_l(e,R)$.
Thus Lemma~\ref{lem:rapid-decay-itap} applies with
$\beta=1$ and
$F(R)=\max\{1,C(1+R)^s\}$.

For the final assertion, property $\mathrm{RD}_2$
implies property $\mathrm{RD}_p$ for every
$p\in(1,2]$ by \cite[Proposition~3.8]{AOP}
(see also \cite[Theorem~4.4]{LY}).
The result for $p\in[2,\infty)$ then follows from
the equivalence of the $p$-ITAP and the $q$-ITAP
for conjugate exponents $p$ and $q$.
\end{proof}

Chen and Wang introduced a subexponential analogue
of property $\mathrm{RD}_p$, called the
$\beta$-Gevrey $\ell^p$-rapid decay property
$(\mathrm{GRD})_{\beta,p}$
\cite[Definition~2.4]{CWGRD}.
Lemma~\ref{lem:rapid-decay-itap} also applies
to groups satisfying this weaker decay condition.

Let $l$ be a proper length function on a discrete group $\Gamma$,
let $p\in[1,\infty)$, and let $0<\beta\leq1$. The group $\Gamma$
is said to have $(\mathrm{GRD})_{\beta,p}$ with respect to $l$ if
there exists a function
\[
F:[0,\infty)\longrightarrow[1,\infty)
\]
such that
\[
\log F(R)=o(R^\beta)
\qquad\text{as }R\to\infty,
\]
and
\[
\|\lambda_p(f)\|
\leq
F(R)\|f\|_{\ell^p(\Gamma)}
\]
whenever $f\in\mathbb C\Gamma$ is supported in
\[
B_l(e,R)=\{g\in\Gamma:l(g)\leq R\}.
\]
By \cite[Proposition 3.5]{CWGRD}, property $\mathrm{RD}_p$ with
respect to a proper length function implies
$(\mathrm{GRD})_{\beta,p}$ for every $0<\beta\leq1$.

\begin{prop}\label{prop:GRD-ITAP}
Let $p\in(1,\infty)$ and $0<\beta\leq1$, and let
$q\in(1,\infty)$ be such that $1/p+1/q=1$.
Suppose that a discrete group $\Gamma$ has
$(\mathrm{GRD})_{\beta,p}$ with respect to a proper length function
$l$ of negative type. Then $\Gamma$ has the $r$-ITAP for every
\[
r\in(1,p]\cup[q,\infty).
\]
In particular, if $p\geq2$, then $\Gamma$ has the $r$-ITAP for every
$r\in(1,\infty)$.
\end{prop}

\begin{proof}
Since $l$ is proper, $\Gamma$ is countable.
Lemma~\ref{lem:rapid-decay-itap} immediately
implies that $\Gamma$ has the $p$-ITAP.
By \cite[Proposition~3.8]{CWGRD},
$\Gamma$ has $(\mathrm{GRD})_{\beta,r}$
with respect to the same length function for every
$r\in(1,p)$.
Hence the same lemma gives the $r$-ITAP for every
$r\in(1,p]$.

The remaining assertions follow from the equivalence
of the $r$-ITAP and the $r'$-ITAP for conjugate
exponents $r$ and $r'$.
\end{proof}

Examples covered by Proposition~\ref{prop:RDp} include groups having
property $\mathrm{RD}_2$ with respect to a length function of negative
type that act properly discontinuously on a finite-dimensional CAT(0)
cube complex with uniformly bounded stabilizers. Such groups include
free groups, finitely generated Coxeter groups, and some small
cancellation groups (cf.~\cite{BN} and the references therein).
Another class is formed by groups acting cocompactly and properly
discontinuously on real or complex hyperbolic space
\cite{FH,Jol1}.

We next study permanence of the $p$-ITAP under subgroups and finite-index supergroups.

Let $G$ be a countable discrete group equipped with a proper right-invariant metric, let $H\leq G$, and equip $H$ with the restricted metric. Choose a set $R$ of representatives of the right cosets of $H$ in $G$, with $e\in R$.

For an $SQ_p$ space $E$, the decomposition $G=\bigsqcup_{r\in R}Hr$ gives an isometric identification
\[
\ell^p(G,E)\cong
\bigoplus_{r\in R}^p\ell^p(H,E),
\]
where each summand is identified with $\ell^p(H,E)$ via right translation.

Define
\[
J:B(\ell^p(H,E))\longrightarrow B(\ell^p(G,E))
\]
by placing the same operator on each summand, and define
\[
E_H:B(\ell^p(G,E))\longrightarrow B(\ell^p(H,E))
\]
by compression to the summand corresponding to $H$.

\begin{lem}
\label{lem:subgroup-maps}
Let $S\subseteq B(E)$ be a $p$-operator space.
The map $J$ is a $p$-complete isometry, and $E_H$
is a $p$-complete contraction satisfying $E_HJ=\mathrm{id}$.
Moreover,
\begin{align*}
J(B_u^p(|H|,S)^H)&\subseteq B_u^p(|G|,S)^G,\\
E_H(B_u^p(|G|,S)^G)&\subseteq B_u^p(|H|,S)^H,\\
J(F_\lambda^p(H)\otimes S)&\subseteq F_\lambda^p(G)\otimes S,\\
E_H(F_\lambda^p(G)\otimes S)&\subseteq F_\lambda^p(H)\otimes S.
\end{align*}
\end{lem}

\begin{proof}
The first assertions follow directly from the direct sum
decomposition and the definition of compression.

Since the metric on $G$ is right invariant, neither
amplification by $J$ nor compression by $E_H$ increases
propagation. Both maps preserve the condition that the
matrix entries belong to $S$, and hence map the
corresponding coefficient Roe spaces into each other.

If $T\in B_u^p(|H|,S)^H$, then the matrix entries of $J(T)$
are given by
\[
J(T)_{s,t}=
\begin{cases}
T_{st^{-1},e},&st^{-1}\in H,\\
0,&st^{-1}\notin H.
\end{cases}
\]
In particular, these entries depend only on $st^{-1}$,
so $J(T)$ is $G$-invariant.

Conversely, compression of a $G$-invariant operator
to $\ell^p(H,E)$ is clearly $H$-invariant.

Finally, for $h\in H$ and $a\in S$,
\[
J(\lambda_h^H\otimes a)=\lambda_h^G\otimes a,
\]
while for $g\in G$,
\[
E_H(\lambda_g^G\otimes a)=
\begin{cases}
\lambda_g^H\otimes a,&g\in H,\\
0,&g\notin H.
\end{cases}
\]
The remaining inclusions follow by linearity and
continuity.
\end{proof}

The matrix-entry argument in the preceding proof does not require the operators to belong to the uniform Roe spaces. In particular, taking $E=\mathbb C$, we obtain
\[
J(CV_p(H))\subseteq CV_p(G),
\qquad
E_H(CV_p(G))\subseteq CV_p(H).
\]

\begin{prop} (cf. \cite[Theorem 3.14]{UZ}) \label{prop:subgp}
Let $G$ be a countable discrete group, let $H$ be a subgroup of $G$, and let
$S\subseteq B(E)$ be a $p$-operator space, where $E$ is an $SQ_p$
space. If $G$ has the $p$-ITAP for $S$, then $H$ has the $p$-ITAP
for $S$.

In particular, the $p$-ITAP and the $p$-operator ITAP both pass to
subgroups.
\end{prop}

\begin{proof}
Let $T\in B_u^p(|H|,S)^H$.
By Lemma~\ref{lem:subgroup-maps} and the hypothesis,
\[
J(T)\in B_u^p(|G|,S)^G
=F_\lambda^p(G)\otimes S.
\]
Applying $E_H$ and using the same lemma, we obtain
\[
T=E_H(J(T))\in F_\lambda^p(H)\otimes S.
\]
Since the reverse inclusion always holds, $H$ has
the $p$-ITAP for $S$.

The scalar and $p$-operator ITAP assertions follow by
taking $S=\mathbb C$ and by considering all
$p$-operator spaces $S\subseteq K(\ell^p)$,
respectively.
\end{proof}

\begin{prop}\label{prop:finite-index}
Let $G$ be a countable discrete group, let $H\leq G$, let $E$ be an $SQ_p$
space, and let $S\subseteq B(E)$ be a $p$-operator space. Suppose
that $H$ has the $p$-ITAP for $S$. For $r\in G$, let $M_{Hr}$
denote Schur multiplication by the kernel
\[
(s,t)\longmapsto\chi_{Hr}(st^{-1}).
\]
Then
\[
M_{Hr}\bigl(B_u^p(|G|,S)^G\bigr)
\subseteq
F_\lambda^p(G)\otimes S
\]
for every $r\in G$.

Consequently, if $H$ has finite index in $G$, then $G$ has the
$p$-ITAP for $S$. Combined with Proposition~\ref{prop:subgp}, it follows that
for a finite-index inclusion $H\leq G$,
\[
G\text{ has the $p$-ITAP for }S
\quad\Longleftrightarrow\quad
H\text{ has the $p$-ITAP for }S.
\]
In particular, the scalar $p$-ITAP and the $p$-operator ITAP are
invariant under passage to finite-index subgroups and finite-index
supergroups, and hence under commensurability.
\end{prop}

\begin{proof}
With respect to the right-coset decomposition
\[
\ell^p(G,E)=
\bigoplus_{q\in R}^p\ell^p(Hq,E),
\]
Schur multiplication by the kernel
$(s,t)\mapsto\mathbf 1_H(st^{-1})$
is precisely the map which retains only the diagonal
blocks. Hence $M_H$ is a $p$-complete contraction.

For $r\in G$, put $L_r=\lambda_r^G\otimes I_E$.
A computation of matrix entries gives
\[
M_{Hr}(A)=M_H(AL_{r^{-1}})L_r
\]
for every $A\in B(\ell^p(G,E))$.
Consequently, each $M_{Hr}$ is also a $p$-complete
contraction.
Since these maps only remove matrix entries, they
preserve $B_u^p(|G|,S)$, and since their defining kernels
depend only on $st^{-1}$, they preserve $G$-invariance.

If $A\in B_u^p(|G|,S)^G$, then its compressions to
the right-coset summands are identical under the
identifications with $\ell^p(H,E)$.
Thus
\[
M_H(A)=J(E_H(A)).
\]
By Lemma~\ref{lem:subgroup-maps} and the assumption
that $H$ has the $p$-ITAP for $S$, it follows that
\[
M_H(A)\in J(F_\lambda^p(H)\otimes S)
\subseteq F_\lambda^p(G)\otimes S.
\]

Now let $T\in B_u^p(|G|,S)^G$.
Right multiplication by $L_{r^{-1}}$ preserves
$B_u^p(|G|,S)$, since it permutes the matrix entries
without changing their membership in $S$ and increases
propagation by at most $d(e,r)$.
Since left translation commutes with right translation,
we also have $TL_{r^{-1}}\in B_u^p(|G|,S)^G$.
The preceding argument therefore gives
\[
M_{Hr}(T)=M_H(TL_{r^{-1}})L_r
\in F_\lambda^p(G)\otimes S.
\]
Here we use
$(\lambda_g^G\otimes a)L_r=\lambda_{gr}^G\otimes a$
for $g\in G$ and $a\in S$.

If $[G:H]<\infty$, then $R$ is finite and
\[
T=\sum_{r\in R}M_{Hr}(T)
\in F_\lambda^p(G)\otimes S.
\]
Hence $G$ has the $p$-ITAP for $S$.

Together with the subgroup permanence result, this
gives the finite-index equivalence.
The scalar and $p$-operator ITAP assertions follow
by the corresponding choices of coefficient spaces,
and commensurability invariance follows from
invariance under group isomorphism.
\end{proof}

\begin{prop}\label{prop:finite-dimensional-coefficients}
Let $\Gamma$ be a countable discrete group. Then the following are equivalent:
\begin{enumerate}
\item $\Gamma$ has the $p$-ITAP;
\item $\Gamma$ has the $p$-ITAP for every finite-dimensional
$p$-operator space $S\subseteq B(E)$, where $E$ is an $SQ_p$ space.
\end{enumerate}
\end{prop}

\begin{proof}
Taking $S=\mathbb C$ gives (ii)$\Rightarrow$(i).

Conversely, suppose that $\Gamma$ has the $p$-ITAP.
Let $(s_j)_{j=1}^m$ be a basis of $S$ and
$(\omega_j)_{j=1}^m$ the corresponding coordinate
functionals. Each $\omega_j$ is $p$-completely bounded.

For $T\in B_u^p(|\Gamma|,S)^\Gamma$, let
$\widehat{\omega}_j(T)$ denote the operator obtained
by applying $\omega_j$ entrywise, as in Section~2.2.
By $\Gamma$-invariance and the assumption of the
$p$-ITAP, we have
\[
\widehat{\omega}_j(T)\in B_u^p(|\Gamma|)^\Gamma
=F_\lambda^p(\Gamma).
\]
Since
\[
T=\sum_{j=1}^m\widehat{\omega}_j(T)\otimes s_j,
\]
we obtain $T\in F_\lambda^p(\Gamma)\otimes S$.
The reverse inclusion is immediate.
\end{proof}

Thus the distinction between the scalar $p$-ITAP and its
operator space version can only be detected by infinite-dimensional
coefficient spaces.

\section{Fubini product and the slice map property}

In this section, we record some properties of Fubini
products and slice maps that will be used in the study
of fixed-point spaces and approximation properties
in the subsequent sections. Throughout this section,
we consider $p$-operator spaces on $L^p$ spaces, so
that the spatial tensor product coincides with the
$p$-operator space injective tensor product. The
results below are $p$-operator space versions of
results in \cite{UZ,KU}.

\begin{defn}
Let $S\subseteq A$ and $T\subseteq B$ be closed
$p$-operator subspaces of $p$-operator spaces on
$L^p$ spaces. The Fubini product of $S$ and $T$
in $A\otimes B$ is
\begin{align*}
F(S,T,A\otimes B)
=\{x\in A\otimes B:\;&
(\varphi\otimes\mathrm{id}_B)(x)\in T
\text{ for all }\varphi\in A^*,\\
&(\mathrm{id}_A\otimes\psi)(x)\in S
\text{ for all }\psi\in B^*\}.
\end{align*}
We say that the triple $(S,T,A\otimes B)$ has the
slice map property if
\[
F(S,T,A\otimes B)=S\otimes T.
\]
We say that $A$ has the slice map property for $B$
if $(A,T,A\otimes B)$ has the slice map property
for every closed $p$-operator subspace $T\subseteq B$.
\end{defn}

Note that
\[
S\otimes T\subseteq F(S,T,A\otimes B)
\subseteq A\otimes B.
\]

We will use the fact that slice maps separate points
of the spatial tensor product. This follows from the realization of the
$p$-operator space injective tensor product in
Section~2.1. In particular, if
$x\in A\otimes B$ satisfies
$(\varphi\otimes\mathrm{id}_B)(x)=0$ for every
$\varphi\in A^*$, then $x=0$, and similarly for
slice maps on the other factor.

It follows directly from the definition that
\[
F(S,T,A\otimes B)
=F(S,B,A\otimes B)\cap F(A,T,A\otimes B).
\]
More generally, for nonempty families of closed
$p$-operator subspaces $(S_\alpha)_\alpha$ of $A$
and $(T_\beta)_\beta$ of $B$, we have
\[
F\left(\bigcap_\alpha S_\alpha,
\bigcap_\beta T_\beta,A\otimes B\right)
=
\bigcap_{\alpha,\beta}
F(S_\alpha,T_\beta,A\otimes B).
\]

\begin{lem} (cf.\ \cite[Theorem~2.15]{UZ})
\label{lem:fubini-kernels}
Let $A,B,C,D$ be $p$-operator spaces on $L^p$ spaces,
and let $(\sigma_\alpha:A\to C)_\alpha$ and
$(\tau_\beta:B\to D)_\beta$ be families of
$p$-completely bounded maps. Then
\[
\begin{split}
&F\left(\bigcap_\alpha\ker\sigma_\alpha,
\bigcap_\beta\ker\tau_\beta,A\otimes B\right)\\
&\qquad=
\left(\bigcap_\alpha
\ker(\sigma_\alpha\otimes\mathrm{id}_B)\right)
\cap
\left(\bigcap_\beta
\ker(\mathrm{id}_A\otimes\tau_\beta)\right).
\end{split}
\]
\end{lem}

\begin{proof}
For $x\in A\otimes B$ and $\varphi\in A^*$, we have
\[
\tau_\beta((\varphi\otimes\mathrm{id}_B)(x))
=
(\varphi\otimes\mathrm{id}_D)
((\mathrm{id}_A\otimes\tau_\beta)(x)).
\]
Since slice maps separate points, all the left slices
of $x$ belong to $\bigcap_\beta\ker\tau_\beta$ if and
only if
\[
x\in\bigcap_\beta
\ker(\mathrm{id}_A\otimes\tau_\beta).
\]
Applying the same argument to the right slices gives
the other condition.
\end{proof}

Let $G$ be a discrete group acting on a $p$-operator
space $A$ by $p$-completely bounded linear maps.
We refer to $A$ as a $G$-$p$-operator space and write
\[
A^G=\{a\in A:ga=a\text{ for every }g\in G\}
\]
for its fixed-point space.

\begin{prop} (cf. \cite[Proposition 2.19 and Corollary 2.20]{UZ}) \label{prop:UZ220}
Let $A$ be a $G$-$p$-operator space on $L^p$ space, and let $B$ be a $H$-$p$-operator space on $L^p$ space.
Then $A\otimes B$ is a $(G\times H)$-$p$-operator space on $L^p$ space, and
\[ F(A^G,B^H,A\otimes B) = (A\otimes B)^{G\times H}. \]
In particular, $A^G\otimes B^H = (A\otimes B)^{G\times H}$ if and only if $(A^G,B^H,A\otimes B)$ has the slice map property.
\end{prop}

\begin{proof}
By functoriality of the tensor product, the actions
of $G$ and $H$ induce an action of $G\times H$
on $A\otimes B$ given by
\[
(g,h)(a\otimes b)=ga\otimes hb.
\]
Since
\[
A^G=\bigcap_{g\in G}\ker(\mathrm{id}_A-g),
\qquad
B^H=\bigcap_{h\in H}\ker(\mathrm{id}_B-h),
\]
Lemma~\ref{lem:fubini-kernels} gives
\begin{align*}
F(A^G,B^H,A\otimes B)
&=(A\otimes B)^{G\times\{e\}}
\cap(A\otimes B)^{\{e\}\times H}\\
&=(A\otimes B)^{G\times H}.
\end{align*}
The final assertion follows from the definition of
the slice map property.
\end{proof}

Let $A$ and $B$ be $p$-operator spaces on $L^p$
spaces. Following \cite[Definition~2.27]{UZ},
we say that $A$ has the $p$-operator approximation
property for $B$ if there exists a net $(\Phi_\alpha)$
of bounded finite-rank maps on $A$ such that
\[
(\Phi_\alpha\otimes\mathrm{id}_B)(x)\longrightarrow x
\]
in norm for every $x\in A\otimes B$.
Recall that bounded finite-rank maps between
$p$-operator spaces are $p$-completely bounded.

Taking $B=K(\ell^p)$ and $B=B(\ell^p)$ gives the
$p$-OAP and the strong $p$-OAP, respectively.

We say that $B$ is matrix stable if, for every
$n\in\mathbb N$, there exists a $p$-completely
bounded isomorphism
\[
B\longrightarrow B\otimes M_n^p
\]
whose inverse is also $p$-completely bounded.
In particular, $K(\ell^p)$ is matrix stable, since
the canonical isometric identification
$\ell^p\cong\ell_n^p\otimes_p\ell^p$ induces
a $p$-completely isometric isomorphism
\[
K(\ell^p)\cong K(\ell^p)\otimes M_n^p.
\]

\begin{thm} \label{thm:OAPslice} (cf. \cite[Theorem 2.28]{UZ} and \cite[Theorem 5.4]{Kraus})
Let $A$ and $B$ be $p$-operator spaces on $L^p$ spaces.
If $A$ has the $p$-operator approximation property
for $B$, then $A$ has the slice map property for $B$.
If $B$ is matrix stable, then the converse also holds.
\end{thm}

\begin{proof}
Suppose first that $A$ has the $p$-operator approximation
property for $B$.
Let $T\subseteq B$ be a closed $p$-operator subspace
and let $x\in F(A,T,A\otimes B)$.

Every bounded finite-rank map $\Phi:A\to A$ has
the form
\[
\Phi(a)=\sum_{j=1}^m\varphi_j(a)a_j
\]
for some $a_j\in A$ and $\varphi_j\in A^*$.
Consequently,
\[
(\Phi\otimes\mathrm{id}_B)(x)
=\sum_{j=1}^m
a_j\otimes(\varphi_j\otimes\mathrm{id}_B)(x)
\in A\odot T.
\]
Applying this to the approximating net of finite-rank
maps shows that $x\in A\otimes T$.
Thus $A$ has the slice map property for $B$.

Conversely, suppose that $A$ has the slice map
property for $B$.
For $x\in A\otimes B$, put
\[
T_x=\overline{\{
(\varphi\otimes\mathrm{id}_B)(x):
\varphi\in A^*\}}.
\]
Then $x\in F(A,T_x,A\otimes B)=A\otimes T_x$.
Since the set of slices
$\{(\varphi\otimes\mathrm{id}_B)(x):\varphi\in A^*\}$
is a dense linear subspace of $T_x$, it follows that
$x$ belongs to the norm closure of
\[
\{(\Phi\otimes\mathrm{id}_B)(x):
\Phi:A\to A\text{ is bounded and finite-rank}\}.
\]

Now suppose, in addition, that $B$ is matrix stable.
Let $x_1,\ldots,x_n\in A\otimes B$ and let
$\varepsilon>0$. Choose a $p$-completely bounded
isomorphism
\[
\pi:B\longrightarrow B\otimes M_n^p
\]
whose inverse is $p$-completely bounded.
Put
\[
z=\sum_{j=1}^n x_j\otimes e_{jj}
\in A\otimes B\otimes M_n^p
\]
and
\[
y=(\mathrm{id}_A\otimes\pi^{-1})(z)
\in A\otimes B.
\]
By the preceding argument, there is a bounded
finite-rank map $\Phi:A\to A$ such that
\[
\|(\Phi\otimes\mathrm{id}_B)(y)-y\|
<
\frac{\varepsilon}{1+\|\pi\|_{pcb}}.
\]
It follows that
\begin{align*}
&\|(\Phi\otimes\mathrm{id}_{B\otimes M_n^p})(z)-z\|\\
&\qquad\leq
\|\pi\|_{pcb}
\|(\Phi\otimes\mathrm{id}_B)(y)-y\|
<\varepsilon.
\end{align*}
Since
\[
(\Phi\otimes\mathrm{id}_{B\otimes M_n^p})(z)-z
=
\sum_{j=1}^n
\big((\Phi\otimes\mathrm{id}_B)(x_j)-x_j\big)
\otimes e_{jj},
\]
compression to each diagonal block gives
\[
\|(\Phi\otimes\mathrm{id}_B)(x_j)-x_j\|
<\varepsilon
\]
for every $1\leq j\leq n$.
This proves that $A$ has the $p$-operator
approximation property for $B$.
\end{proof}

In particular, since $K(\ell^p)$ is matrix stable,
a $p$-operator space $A$ on an $L^p$ space has the
$p$-OAP if and only if it has the slice map property
for $K(\ell^p)$.

\section{The $p$-AP and invariant uniform Roe algebras}

In this section, we first characterize the $p$-AP and $p$-weak amenability in terms of approximation properties of the inclusion of the reduced $L^p$ group algebra into the algebra of $p$-convoluters.
We then study the consequences of these results and of multiplier approximation for invariant uniform Roe algebras, and obtain further characterizations of the $p$-AP.

\subsection{Approximation properties of the convolver inclusion}

Let $\Gamma$ be a discrete group and let $1<p<\infty$.
We use the $p$-completely contractive coproduct
\[
    \Delta_p:
    F_\lambda^p(\Gamma)
    \longrightarrow
    F_\lambda^p(\Gamma)\otimes F_\lambda^p(\Gamma),
\]
where the tensor product is spatial, determined by
\[
    \Delta_p(\lambda_p(s))
    =\lambda_p(s)\otimes\lambda_p(s)
    \qquad(s\in\Gamma);
\]
see \cite[Section~4]{ALR}.

The following lemma extends the multiplier extraction
argument in \cite[Proposition~5.3]{ALR} to maps from
the reduced group algebra into the convolver algebra.

\begin{lem} \label{lem:multiplier-extraction}
Let $\Gamma$ be a discrete group and let $1<p<\infty$.
For a $p$-completely bounded map
\[
    T:F_\lambda^p(\Gamma)\longrightarrow CV_p(\Gamma),
\]
define
\[
    u_T(s)=\delta_s^*(T(\lambda_p(s))\delta_e)
    \qquad(s\in\Gamma).
\]
Then $u_T\in M_{cb}A_p(\Gamma)$ and
\[
    \|u_T\|_{M_{cb}A_p(\Gamma)}
    \leq \|T\|_{pcb}.
\]
If $T$ has finite rank, then $u_T\in A_p(\Gamma)$.
\end{lem}

\begin{proof}
Define contractions
\[
    W:\ell^p(\Gamma)\longrightarrow
    \ell^p(\Gamma\times\Gamma),
    \qquad W\delta_t=\delta_t\otimes\delta_t,
\]
and
\[
    Q:\ell^p(\Gamma\times\Gamma)
    \longrightarrow\ell^p(\Gamma),
    \qquad
    Q(\delta_r\otimes\delta_t)=
    \begin{cases}
        \delta_t,&r=t,\\
        0,&r\neq t.
    \end{cases}
\]
The map
\[
    R:B(\ell^p(\Gamma\times\Gamma))
    \longrightarrow B(\ell^p(\Gamma)),
    \qquad R(a)=QaW,
\]
is $p$-completely contractive.

For $b\in CV_p(\Gamma)$, right-translation invariance gives
$b_{st,t}=b_{s,e}$ for all $s,t\in\Gamma$. Hence
\[
    R(\lambda_p(s)\otimes b)\delta_t
    =b_{st,t}\delta_{st}
    =b_{s,e}\lambda_p(s)\delta_t.
\]
In particular,
\[
    R(\lambda_p(s)\otimes b)
    =b_{s,e}\lambda_p(s).
\]
By density and continuity, it follows that
\[
    R\big(F_\lambda^p(\Gamma)\otimes CV_p(\Gamma)\big)
    \subseteq F_\lambda^p(\Gamma).
\]

The map
\[
    R\circ(\mathrm{id}\otimes T)\circ\Delta_p
\]
is thus a $p$-completely bounded map on
$F_\lambda^p(\Gamma)$ of $p$-completely bounded norm
at most $\|T\|_{pcb}$, and it sends $\lambda_p(s)$ to
$u_T(s)\lambda_p(s)$. By the scalar case of \cite[Theorem~4.1]{ALR},
$u_T\in M_{cb}A_p(\Gamma)$ with the asserted norm bound.
Writing $\overline{M}_u$ for the multiplier map on
$F_\lambda^p(\Gamma)$ associated to
$u\in M_{cb}A_p(\Gamma)$, we have
\begin{equation}\label{eq:multiplier-extraction}
    \overline{M}_{u_T}
    =R\circ(\mathrm{id}\otimes T)\circ\Delta_p.
\end{equation}

If $T$ has finite rank, write
\[
    T(a)=\sum_{j=1}^n f_j(a)b_j,
\]
where $f_j\in F_\lambda^p(\Gamma)^*$ and
$b_j\in CV_p(\Gamma)$. Then
\[
    u_T(s)=
    \sum_{j=1}^n f_j(\lambda_p(s))
    \delta_s^*(b_j\delta_e).
\]
Since each function
$s\mapsto f_j(\lambda_p(s))$ is bounded and
$b_j\delta_e\in\ell^p(\Gamma)$, we have
$u_T\in\ell^p(\Gamma)$.

Finally, $\ell^p(\Gamma)\subseteq A_p(\Gamma)$, since for
$h\in\ell^p(\Gamma)$, putting $\check h(t)=h(t^{-1})$,
we have
\[
    h(s)=\delta_e^*(\lambda_p(s)\check h).
\]
This proves the final assertion.
\end{proof}

\begin{thm} \label{thm:convolver-inclusion-oap}
Let $\Gamma$ be a discrete group and let $1<p<\infty$.
For the canonical inclusion
\[
    \iota_p:F_\lambda^p(\Gamma)\longrightarrow CV_p(\Gamma),
\]
the following are equivalent:
\begin{enumerate}
    \item[(i)] $\Gamma$ has the $p$-AP.
    \item[(ii)] $\iota_p$ has the $p$-OAP.
    \item[(iii)] $\iota_p$ has the strong $p$-OAP.
\end{enumerate}
\end{thm}

\begin{proof}
The argument adapts the proof of
\cite[Theorem~5.5]{ALR}, using the preceding
multiplier extraction lemma.

If $\Gamma$ has the $p$-AP, then
$F_\lambda^p(\Gamma)$ has the strong $p$-OAP by
\cite[Theorem~5.5]{ALR}. Postcomposing the corresponding
finite-rank approximating maps with $\iota_p$ gives
(iii). Clearly, (iii) implies (ii).

Suppose that (ii) holds, and let
\[
    T_\alpha:F_\lambda^p(\Gamma)\longrightarrow CV_p(\Gamma)
\]
be finite-rank maps witnessing the $p$-OAP of $\iota_p$.
Put $u_\alpha=u_{T_\alpha}$ as in
Lemma~\ref{lem:multiplier-extraction}. Then
$u_\alpha\in A_p(\Gamma)$.

We show that $u_\alpha$ converges to $1$ in the
$\sigma(M_{cb}A_p(\Gamma),Q_{pcb}(\Gamma))$ topology.

For an index set $J$, write
\[
    B_J=B(\ell^p(J)),\qquad
    K_J=K(\ell^p(J)),\qquad
    N_J=\ell^p(J)\widehat{\otimes}_\pi\ell^{p'}(J),
\]
where $N_J$ is the canonical predual of $B_J$ and $p'$ is
the conjugate exponent of $p$.

By \cite[Proposition~4.3]{ALR}, every element of
$Q_{pcb}(\Gamma)$ has the form
\[
    \omega_{a,\varphi}(u)
    =
    \left\langle
        (\mathrm{id}_{K_I}\otimes\overline{M}_u)(a),
        \varphi
    \right\rangle,
\]
where $I$ is a countably infinite index set,
$a\in K_I\otimes F_\lambda^p(\Gamma)$ and
$\varphi\in N_{I\times\Gamma}$.

Fix such $a$ and $\varphi$. Put
\[
    Y=(\mathrm{id}_{K_I}\otimes\Delta_p)(a)
    \in K_I\otimes F_\lambda^p(\Gamma)
        \otimes F_\lambda^p(\Gamma),
\]
and let $\psi\in N_{I\times\Gamma\times\Gamma}$ be the
nuclear functional obtained from $\varphi$ by
precomposing with the amplified compression
\[
    Z\longmapsto
    (1_I\otimes Q)Z(1_I\otimes W).
\]
Here $Q$ and $W$ are the maps defined in the proof of
Lemma~\ref{lem:multiplier-extraction}.

It follows from \eqref{eq:multiplier-extraction} that
\[
    \omega_{a,\varphi}(u_\alpha)
    =
    \left\langle
        (\mathrm{id}\otimes\mathrm{id}\otimes T_\alpha)(Y),
        \psi
    \right\rangle.
\]

Put $J=I\times\Gamma$, and regard $Y$ as an element of
$B_J\otimes F_\lambda^p(\Gamma)$.
Give $N_{J\times\Gamma}$ the right Banach $K_J$-module
structure determined by
\[
    \langle Z,\eta\cdot x\rangle
    =
    \langle(x\otimes 1)Z,\eta\rangle
    \qquad
    (Z\in B_{J\times\Gamma},\ x\in K_J).
\]
If $P_F$ denotes the coordinate projection associated to a
finite subset $F\subseteq J$, then
\[
    \eta\cdot P_F\longrightarrow\eta
\]
in nuclear norm. This follows first for elementary nuclear
tensors from the strong convergence of the corresponding
coordinate projections on $\ell^{p'}(J\times\Gamma)$, and
then by density.

Thus $N_{J\times\Gamma}$ is an essential right
Banach $K_J$-module. Since $K_J$ has a contractive
approximate identity, the Cohen--Hewitt
factorization theorem \cite[Theorem~2.9.24]{Dales00} gives
$\psi_0\in N_{J\times\Gamma}$ and $x\in K_J$ such that
$\psi=\psi_0\cdot x$. Consequently,
\[
    \omega_{a,\varphi}(u_\alpha)
    =
    \left\langle
        (\mathrm{id}\otimes T_\alpha)((x\otimes 1)Y),
        \psi_0
    \right\rangle.
\]
Moreover,
\[
    (x\otimes 1)Y\in K_J\otimes F_\lambda^p(\Gamma).
\]

Although $J$ need not be countable, every element of
$K_J\otimes F_\lambda^p(\Gamma)$ is supported on a
countable subset of $J$ in the first tensor factor.
Indeed, for every $b\in K_J$, compactness of $b$
and its adjoint, together with strong convergence
of the finite-coordinate projections on
$\ell^p(J)$ and $\ell^{p'}(J)$, gives
\[
\|b-P_FbP_F\|
\leq \|(1-P_F)b\|+\|b(1-P_F)\|
\longrightarrow 0
\]
as $F$ increases over the finite subsets of $J$.

By taking a sequence of algebraic tensor
approximations, it follows that every element of
$K_J\otimes F_\lambda^p(\Gamma)$ belongs to
$K(\ell^p(J_0))\otimes F_\lambda^p(\Gamma)$ for
some countable subset $J_0\subseteq J$, where
$K(\ell^p(J_0))$ is embedded in $K_J$ by extension
by zero.

Identifying $\ell^p(J_0)$ with a coordinate subspace
of $\ell^p$, the $p$-OAP of $\iota_p$ therefore gives

\[
    (\mathrm{id}\otimes T_\alpha)((x\otimes 1)Y)
    \longrightarrow
    (\mathrm{id}\otimes\iota_p)((x\otimes 1)Y)
\]
in norm.

It follows that
\[
    \omega_{a,\varphi}(u_\alpha)
    \longrightarrow
    \omega_{a,\varphi}(u_{\iota_p})
    =
    \omega_{a,\varphi}(1).
\]
Since the functionals $\omega_{a,\varphi}$ exhaust
$Q_{pcb}(\Gamma)$, we conclude that $\Gamma$ has the
$p$-AP.
\end{proof}

We also obtain a bounded approximation characterization of
$p$-weak amenability. Write
$\Lambda_p^{\mathrm{inc}}(\Gamma)$ for the infimum of all
$C\geq 1$ for which there exists a net of finite-rank maps
\[
    T_\alpha:F_\lambda^p(\Gamma)\longrightarrow CV_p(\Gamma)
\]
such that
\[
    \sup_\alpha\|T_\alpha\|_{pcb}\leq C
    \quad\text{and}\quad
    T_\alpha(a)\longrightarrow a
    \qquad(a\in F_\lambda^p(\Gamma)).
\]
As usual, the infimum is taken to be infinity if no such
net exists.

\begin{prop} \label{prop:convolver-inclusion-cbap}
For every discrete group $\Gamma$ and every $1<p<\infty$,
\[
    \Lambda_p^{\mathrm{inc}}(\Gamma)=\Lambda_p(\Gamma).
\]
\end{prop}

\begin{proof}
By \cite[Theorem~5.4]{ALR}, the $p$-weak amenability
constant of $\Gamma$ is equal to
$\Lambda_{pcb}(F_\lambda^p(\Gamma))$.
Postcomposing finite-rank approximating maps on
$F_\lambda^p(\Gamma)$ with the inclusion into $CV_p(\Gamma)$
therefore gives
\[
    \Lambda_p^{\mathrm{inc}}(\Gamma)\leq\Lambda_p(\Gamma).
\]

Conversely, suppose that $(T_\alpha)$ is a net as above,
with $\sup_\alpha\|T_\alpha\|_{pcb}\leq C$.
By Lemma~\ref{lem:multiplier-extraction}, the associated
functions $u_\alpha=u_{T_\alpha}$ belong to $A_p(\Gamma)$
and satisfy
\[
    \|u_\alpha\|_{M_{cb}A_p(\Gamma)}\leq C.
\]
Moreover, for every $s\in\Gamma$,
\[
    u_\alpha(s)
    =
    \delta_s^*(T_\alpha(\lambda_p(s))\delta_e)
    \longrightarrow 1.
\]
Since finitely supported functions are dense in
$A_p(\Gamma)$, the uniform multiplier norm bound and
pointwise convergence imply that $(u_\alpha)$ is an
approximate identity for $A_p(\Gamma)$. Thus
$\Lambda_p(\Gamma)\leq C$.
Taking the infimum over all such $C$ proves the assertion.
\end{proof}

More generally, let $V$ be a closed subspace of $CV_p(\Gamma)$
containing $F_\lambda^p(\Gamma)$, with its inherited
$p$-operator space structure. By composing finite-rank
approximating maps with the canonical inclusions
\[
    F_\lambda^p(\Gamma)\hookrightarrow V
    \hookrightarrow CV_p(\Gamma),
\]
we see that $\Gamma$ has the $p$-AP if and only if the
inclusion of $F_\lambda^p(\Gamma)$ into $V$ has the
$p$-OAP, equivalently the strong $p$-OAP. Similarly, the
optimal uniform bound for finite-rank point-norm
approximations of this inclusion is $\Lambda_p(\Gamma)$.

\subsection{Invariant uniform Roe algebras
and coefficient spaces}

For a countable discrete group $\Gamma$, we have
the canonical inclusions
\[
F_\lambda^p(\Gamma)
\subseteq B_u^p(|\Gamma|)^\Gamma
\subseteq CV_p(\Gamma).
\]
We now apply the results of the preceding
subsection to these inclusions and examine the
relation between the $p$-AP and invariant
translation approximation properties.

\begin{lem} (cf. \cite[Lemma 3.1]{Zach}) \label{lem:Zach}
Suppose that $\Gamma$ is a countable exact discrete group, $E$ is an $SQ_p$ space, and $S\subseteq B(E)$ is a $p$-operator space.
Then \[ B^p_u(|\Gamma|,S)^\Gamma = (B_u^p(|\Gamma|) \otimes S)^\Gamma. \]
\end{lem}

\begin{proof}
The inclusion from right to left follows from the
canonical equivariant inclusion
\[
B_u^p(|\Gamma|)\otimes S
\subseteq B_u^p(|\Gamma|,S).
\]

Conversely, let $a\in B_u^p(|\Gamma|,S)^\Gamma$.
Since $\Gamma$ is exact, $|\Gamma|$ has property A.
Given $\varepsilon>0$, choose a finite-propagation
kernel $k$ as in
Lemma~\ref{lem:local-schur-approximation}
such that
\[
\|M_k(a)-a\|<\varepsilon.
\]
Choose $R\geq 0$ such that $k$ has propagation at most
$R$, and put $F=B(e,R)$.

By $\Gamma$-invariance, $a_{s,t}=a_{st^{-1},e}$ for
all $s,t\in\Gamma$. Since $k(s,t)=0$ whenever
$st^{-1}\notin F$, we have
\[
M_k(a)=\sum_{g\in F}M_k(\lambda_g)\otimes a_{g,e}
\in B_u^p(|\Gamma|)\odot S.
\]
Since $\varepsilon>0$ was arbitrary, it follows that
$a\in B_u^p(|\Gamma|)\otimes S$.
Together with $\Gamma$-invariance, this gives the
desired inclusion.
\end{proof}

\begin{thm} \label{thm:pAPcoeff}
Let $\Gamma$ be a countable discrete group with
the $p$-AP, where $1<p<\infty$.
Then, for every $SQ_p$ space $E$ and every closed
$p$-operator space $S\subseteq B(E)$, the following
hold.
\begin{enumerate}
\item
\[
B_u^p(|\Gamma|,S)
=
\{a\in B_u^p(|\Gamma|,B(E)):
a_{s,t}\in S\text{ for all }s,t\in\Gamma\}.
\]

\item
\[
B_u^p(|\Gamma|,S)^\Gamma
=
(B_u^p(|\Gamma|)\otimes S)^\Gamma
=
F_\lambda^p(\Gamma)\otimes S.
\]
\end{enumerate}
In particular, $\Gamma$ has the $p$-operator ITAP.
\end{thm}

\begin{proof}
By Corollary~\ref{cor:multiplier-transfer},
applied to the identity homomorphism on $\Gamma$,
there is a net $(u_\alpha)$ of finitely supported
elements of $A_p(\Gamma)$ such that
\[
\widehat m_{u_\alpha}(a)\longrightarrow a
\]
in norm for every $a\in B_u^p(|\Gamma|,B(E))$.

(i) Suppose that $a\in B_u^p(|\Gamma|,B(E))$
has all its matrix entries in $S$.
Since each $u_\alpha$ is finitely supported,
$\widehat m_{u_\alpha}(a)$ has finite propagation
and all its entries belong to $S$. Hence
\[
\widehat m_{u_\alpha}(a)\in\mathbb C[|\Gamma|,S].
\]
Taking the norm limit gives
$a\in B_u^p(|\Gamma|,S)$.
The reverse inclusion is immediate.

(ii) Let $a\in B_u^p(|\Gamma|,S)^\Gamma$.
By $\Gamma$-invariance, we have
$a_{s,t}=a_{st^{-1},e}$ for all $s,t\in\Gamma$.
Consequently,
\[
\widehat m_{u_\alpha}(a)
=
\sum_{g\in\operatorname{supp}(u_\alpha)}
u_\alpha(g)\lambda_g\otimes a_{g,e}
\in F_\lambda^p(\Gamma)\odot S.
\]
Taking the norm limit gives
$a\in F_\lambda^p(\Gamma)\otimes S$.
The conclusion now follows from the inclusions
\[
F_\lambda^p(\Gamma)\otimes S
\subseteq
(B_u^p(|\Gamma|)\otimes S)^\Gamma
\subseteq B_u^p(|\Gamma|,S)^\Gamma.
\]
\end{proof}

\begin{cor} \label{cor:APcoeff}
Let $\Gamma$ be a countable discrete group and let
$p\in(1,\infty)$.
If $\Gamma$ is $p$-weakly amenable, then
\[
B_u^p(|\Gamma|,S)^\Gamma
=
(B_u^p(|\Gamma|)\otimes S)^\Gamma
=
F_\lambda^p(\Gamma)\otimes S
\]
for every $SQ_p$ space $E$ and every
$p$-operator space $S\subseteq B(E)$.

If $\Gamma$ has the Haagerup--Kraus approximation
property, then the same conclusion holds for every
$p\in(1,\infty)$.
\end{cor}

\begin{proof}
Both assertions follow from
Theorem~\ref{thm:pAPcoeff}(ii) and the implications
between approximation properties recorded in
Remark~\ref{rem:pAP}.
\end{proof}

\begin{thm} \label{thm:invariant-roe-oap}
Let $\Gamma$ be a countable discrete group and let
$1<p<\infty$. The following are equivalent:
\begin{enumerate}
    \item[(i)] $\Gamma$ has the $p$-AP.
    \item[(ii)] $B_u^p(|\Gamma|)^\Gamma$ has the $p$-OAP.
    \item[(iii)] $B_u^p(|\Gamma|)^\Gamma$ has the
    strong $p$-OAP.
\end{enumerate}
Moreover,
\[
    \Lambda_{pcb}(B_u^p(|\Gamma|)^\Gamma)
    =\Lambda_p(\Gamma).
\]
\end{thm}

\begin{proof}
If $\Gamma$ has the $p$-AP, then it has the
$p$-ITAP by Theorem~\ref{thm:pAPcoeff}.
Consequently,
\[
B_u^p(|\Gamma|)^\Gamma=F_\lambda^p(\Gamma),
\]
which has the strong $p$-OAP by
\cite[Theorem~5.5]{ALR}.
Clearly, (iii) implies (ii).

Conversely, if $B_u^p(|\Gamma|)^\Gamma$ has the
$p$-OAP, then restricting its finite-rank
approximating maps to $F_\lambda^p(\Gamma)$
and composing with the inclusion into
$CV_p(\Gamma)$ shows that the inclusion
$F_\lambda^p(\Gamma)\hookrightarrow CV_p(\Gamma)$
has the $p$-OAP.
Theorem~\ref{thm:convolver-inclusion-oap}
therefore gives (i).

The same restriction and composition argument,
together with
Proposition~\ref{prop:convolver-inclusion-cbap},
gives
\[
\Lambda_p(\Gamma)
\leq\Lambda_{pcb}(B_u^p(|\Gamma|)^\Gamma).
\]
If $\Lambda_p(\Gamma)<\infty$, then $\Gamma$
has the $p$-AP and hence the $p$-ITAP.
Thus equality follows from
\cite[Theorem~5.4]{ALR}.
If $\Lambda_p(\Gamma)=\infty$, equality follows
from the preceding inequality.
\end{proof}

\begin{cor} \label{cor:noncomplementation}
Let $\Gamma$ be an exact countable discrete group and let
$1<p<\infty$. Suppose that
\[
    P:B_u^p(|\Gamma|)\longrightarrow CV_p(\Gamma)
\]
is a $p$-completely bounded map satisfying
\[
    P(a)=a\qquad(a\in F_\lambda^p(\Gamma)).
\]
Then
\[
    \Lambda_p(\Gamma)\leq\|P\|_{pcb}.
\]
In particular, if $\Gamma$ is not $p$-weakly amenable,
then $B_u^p(|\Gamma|)^\Gamma$ is not the range of a
$p$-completely bounded projection on $B_u^p(|\Gamma|)$.
\end{cor}

\begin{proof}
Since $\Gamma$ is exact, $|\Gamma|$ has property A,
and hence $B_u^p(|\Gamma|)$ has the $p$-CCAP.
Let $(\Phi_\alpha)$ be a net of finite-rank
$p$-complete contractions on $B_u^p(|\Gamma|)$
converging pointwise in norm to the identity.

The maps
\[
    P\circ\Phi_\alpha|_{F_\lambda^p(\Gamma)}
\]
are finite-rank maps from $F_\lambda^p(\Gamma)$ into
$CV_p(\Gamma)$, converge pointwise in norm to the
inclusion, and have $p$-completely bounded norms at most
$\|P\|_{pcb}$. Proposition~
\ref{prop:convolver-inclusion-cbap} gives the desired bound.
The final assertion follows since any projection onto
$B_u^p(|\Gamma|)^\Gamma$ fixes $F_\lambda^p(\Gamma)$
pointwise.
\end{proof}

\begin{thm} (cf. \cite[Theorem 3.2]{Zach}) \label{thm:operator-itap-entry}
Let $\Gamma$ be a countable discrete group and let
$1<p<\infty$. Then $\Gamma$ has the $p$-AP if and only
if it has the $p$-operator ITAP and satisfies the following
condition: for every closed $p$-operator space
$S\subseteq K(\ell^p)$,
\[
    B_u^p(|\Gamma|,S)
    =\{a\in B_u^p(|\Gamma|,K(\ell^p)):
    \, a_{s,t}\in S \; \text{ for all }s,t\in\Gamma\}.
\]
\end{thm}

\begin{proof}
If $\Gamma$ has the $p$-AP, then it has the
$p$-operator ITAP by
Theorem~\ref{thm:pAPcoeff}(ii).
The coefficient-entry condition follows from
Theorem~\ref{thm:pAPcoeff}(i), applied with
$E=\ell^p$.

Conversely, suppose that $\Gamma$ has the $p$-operator ITAP
and satisfies the stated coefficient-entry condition.
Let $S\subseteq K(\ell^p)$ be a closed subspace and let
\[
    x\in F(F_\lambda^p(\Gamma),S,
    F_\lambda^p(\Gamma)\otimes K(\ell^p)).
\]
Then
\[
    x\in B_u^p(|\Gamma|,K(\ell^p))^\Gamma.
\]
Moreover, every matrix entry of $x$ belongs to $S$,
since
\[
    x_{s,t}=(\omega_{s,t}\otimes\mathrm{id})(x),
\]
where
$\omega_{s,t}(a)=\delta_s^*(a\delta_t)$.

By the coefficient-entry condition, we have
$x\in B_u^p(|\Gamma|,S)^\Gamma$. The $p$-operator ITAP
therefore gives
\[
    x\in F_\lambda^p(\Gamma)\otimes S.
\]
Thus $F_\lambda^p(\Gamma)$ has the slice map property
for $K(\ell^p)$, and hence the $p$-OAP by
Theorem~\ref{thm:OAPslice}. It follows from
\cite[Theorem~5.5]{ALR} that $\Gamma$ has the $p$-AP.
\end{proof}

\subsection{Exact groups and fixed-point
tensor identities}

For exact groups, the coefficient-entry condition
in Theorem~\ref{thm:operator-itap-entry} follows
from property A. We obtain additional
characterizations of the $p$-AP in this setting.

\begin{thm} \label{thm:exact-group-characterizations}
Let $\Gamma$ be an exact countable discrete group and let
$1<p<\infty$. The following are equivalent:
\begin{enumerate}
    \item[(i)] $\Gamma$ has the $p$-AP.

    \item[(ii)] $\Gamma$ has the $p$-operator ITAP.

    \item[(iii)] For every closed $p$-operator space
    $S\subseteq K(\ell^p)$,
    \[
        B_u^p(|\Gamma|,S)^\Gamma
        =
        (B_u^p(|\Gamma|)\otimes S)^\Gamma
        =
        F_\lambda^p(\Gamma)\otimes S.
    \]

    \item[(iv)] The equality in (iii) holds for every
    $SQ_p$ space $E$ and every closed $p$-operator space
    $S\subseteq B(E)$.

    \item[(v)] $B_u^p(|\Gamma|)^\Gamma$ has the $p$-OAP.

    \item[(vi)] $B_u^p(|\Gamma|)^\Gamma$ has the strong
    $p$-OAP.

    \item[(vii)] For every closed $p$-operator space
    $S\subseteq K(\ell^p)$,
    \[
        B_u^p(|\Gamma|)^\Gamma\otimes S
        =
        (B_u^p(|\Gamma|)\otimes S)^\Gamma.
    \]
\end{enumerate}
\end{thm}

\begin{proof}
Theorem~\ref{thm:pAPcoeff}(ii) gives (i)$\Rightarrow$(iv), while
(iv)$\Rightarrow$(iii)$\Rightarrow$(ii) are immediate.
Since $\Gamma$ is exact, the coefficient-entry condition
in Theorem~\ref{thm:operator-itap-entry} holds by
Theorem~\ref{thm:functor}. Thus (ii) implies (i).

The equivalence of (i), (v) and (vi) follows from
Theorem~\ref{thm:invariant-roe-oap}.
Finally, (iii) implies (vii), since (ii) implies the scalar
$p$-ITAP.

It remains to show that (vii) implies (v).
For brevity, put
\[
A=B_u^p(|\Gamma|),\qquad K=K(\ell^p).
\]
Since $\Gamma$ is exact, $A$ is $p$-nuclear by
Theorem~\ref{Apnuc} and hence has the $p$-OAP.
By Theorem~\ref{thm:OAPslice}, it therefore has
the slice map property for $K$.

Let $S\subseteq K$ be a closed subspace and let
\[
x\in F(A^\Gamma,S,A^\Gamma\otimes K).
\]
By restricting functionals on $A$ to $A^\Gamma$,
we obtain
\[
x\in F(A,S,A\otimes K)=A\otimes S.
\]
Since $x$ is $\Gamma$-invariant, (vii) gives
\[
x\in(A\otimes S)^\Gamma=A^\Gamma\otimes S.
\]
Thus $A^\Gamma$ has the slice map property for $K$
and hence has the $p$-OAP by
Theorem~\ref{thm:OAPslice}.
\end{proof}

\begin{cor}
Let $\Gamma$ be an exact countable discrete group. Then
\begin{enumerate}
\item For $p,q\in(1,\infty)$ with $1/p+1/q=1$, $\Gamma$ has the $p$-operator ITAP if and only if it has the $q$-operator ITAP.
\item If $1<p\leq q\leq 2$ or $2\leq q\leq p<\infty$, and $\Gamma$ has the $q$-operator ITAP, then it has the $p$-operator ITAP.
\end{enumerate}
\end{cor}

\begin{proof}
(i) follows from Theorem \ref{thm:exact-group-characterizations} and \cite[Proposition 2.3]{Ver}.
(ii) follows from Theorem \ref{thm:exact-group-characterizations} and \cite[Proposition 6.3]{ALR}.
\end{proof}

\begin{cor} \label{cor:higher-rank}
Let $G$ be a connected simple Lie group with finite
center and real rank greater than one, and let
$\Gamma$ be a lattice in $G$.
Then, for every $p\in(1,\infty)$, the following hold.
\begin{enumerate}
\item $\Gamma$ does not have the $p$-operator ITAP.
Moreover, no discrete group commensurable with
$\Gamma$ has the $p$-operator ITAP.

\item There exists a closed $p$-operator space
$S\subseteq K(\ell^p)$ such that
\[
B_u^p(|\Gamma|)^\Gamma\otimes S
\subsetneq
(B_u^p(|\Gamma|)\otimes S)^\Gamma.
\]

\item $B_u^p(|\Gamma|)^\Gamma$ does not have the
$p$-OAP and is not the range of a $p$-completely
bounded projection on $B_u^p(|\Gamma|)$.
\end{enumerate}
In particular, these conclusions apply to
$\Gamma=\mathrm{SL}(n,\mathbb Z)$ for every $n\geq3$.
\end{cor}

\begin{proof}
By \cite[Theorem~1.5]{Ver}, $\Gamma$ does not have the $p$-AP
for any $p\in(1,\infty)$. It remains to note that $\Gamma$ is exact.
The adjoint representation
\[
\operatorname{Ad}:G\longrightarrow GL(\mathfrak g)
\]
has kernel $Z(G)$. Hence $\operatorname{Ad}(\Gamma)$ is a countable
linear group and is exact by \cite[Theorem~6]{GHW}. Since
\[
\ker(\operatorname{Ad}|_\Gamma)
=
\Gamma\cap Z(G)
\]
is finite and exactness
is preserved under extensions with amenable normal
subgroups (cf.~\cite{BO}), $\Gamma$ is exact.

The failure of the $p$-operator ITAP and of the
fixed-point tensor identity follows from
Theorem~\ref{thm:exact-group-characterizations}.
The assertion concerning commensurable groups
follows from Proposition~\ref{prop:finite-index}.

By Theorem~\ref{thm:invariant-roe-oap},
$B_u^p(|\Gamma|)^\Gamma$ does not have the $p$-OAP.
Since $\Gamma$ does not have the $p$-AP, it is not
$p$-weakly amenable. Hence the noncomplementation
assertion follows from
Corollary~\ref{cor:noncomplementation}.

Finally, take $G=\mathrm{SL}(n,\mathbb R)$ and
$\Gamma=\mathrm{SL}(n,\mathbb Z)$ for $n\geq3$.
\end{proof}

For $p=2$, the failure of the fixed-point tensor equality
was established by Katsura and Uuye \cite{KU}.
The preceding corollary extends this phenomenon to every
$p\in(1,\infty)$.

\section{The $p$-ITAP for group extensions and product groups}

In this section, we study the $p$-ITAP for direct
products and group extensions. We first establish
some fixed-point and slice map results for product
groups and coefficient spaces carrying group actions.
We then prove permanence of the $p$-ITAP under
group extensions with quotient groups having the
$p$-AP, with applications to semidirect products
and permanence of the $p$-AP.

\subsection{Direct product fixed-point results}

Let $G$ and $H$ be countable discrete groups equipped
with proper right-invariant metrics $d_G$ and $d_H$,
respectively. We equip $G\times H$ with the proper
right-invariant metric
\[
d_{G\times H}((g,h),(g',h'))
=
\max\{d_G(g,g'),d_H(h,h')\}.
\]
With this choice, we have
$|G\times H|=|G|\times|H|$,
where the product metric is as in Section~2.2
(cf.\ \cite[Lemma~3.15]{UZ}).

\begin{prop} (cf. \cite[Proposition 3.16]{UZ}) \label{prop:factor}
Let $G$ and $H$ be countable discrete groups. If $G\times H$ has the $p$-ITAP, then so do both $G$ and $H$.
Moreover, the triple $( B^p_u(|G|)^G,B^p_u(|H|)^H,B^p_u(|G|)\otimes B^p_u(|H|) )$ has the slice map property.
\end{prop}

\begin{proof}
By subgroup permanence, both $G$ and $H$ have the
$p$-ITAP.

Under the identification
$\ell^p(G\times H)=\ell^p(G)\otimes_p\ell^p(H)$,
Proposition~\ref{thm:incl} gives a natural
equivariant inclusion
\[
B_u^p(|G|)\otimes B_u^p(|H|)
\subseteq B_u^p(|G\times H|).
\]
Consequently,
\begin{align*}
&(B_u^p(|G|)\otimes B_u^p(|H|))^{G\times H}\\
&\qquad\subseteq B_u^p(|G\times H|)^{G\times H}\\
&\qquad=F_\lambda^p(G\times H)\\
&\qquad=F_\lambda^p(G)\otimes F_\lambda^p(H)\\
&\qquad=B_u^p(|G|)^G\otimes B_u^p(|H|)^H,
\end{align*}
where the penultimate equality follows from
$\lambda_{(g,h)}=\lambda_g\otimes\lambda_h$,
and the last equality follows from the $p$-ITAP
of $G$ and $H$.

Since the reverse inclusion is immediate, we obtain
\[
(B_u^p(|G|)\otimes B_u^p(|H|))^{G\times H}
=B_u^p(|G|)^G\otimes B_u^p(|H|)^H.
\]
The assertion about the slice map property now
follows from Proposition~\ref{prop:UZ220}.
\end{proof}

We next consider fixed-point spaces when the
coefficient space carries a group action.

\begin{thm} (cf. \cite[Theorem 3.17]{UZ}) \label{thm:UZ317}
Let $G$ be a countable discrete group with the $p$-AP, and let $S\subseteq B(\ell^p)$ be a $p$-operator space equipped with an action of a discrete group $H$ by $p$-completely bounded maps. We equip $B_u^p(|G|,S)$ with the action of $G\times H$ obtained from right translation by $G$ and entrywise application of the action of $H$ on $S$. Then we have
\[ F^p_\lambda(G)\otimes S^H = B^p_u(|G|,S)^{G\times H}. \]
\end{thm}

\begin{proof}
Since $G$ has the $p$-AP, Theorem~\ref{thm:pAPcoeff}(ii) gives
\[
B_u^p(|G|,S)^G=F_\lambda^p(G)\otimes S.
\]
This equality is $H$-equivariant, where $H$ acts
on the coefficient space $S$. Consequently,
\[
B_u^p(|G|,S)^{G\times H}
=(F_\lambda^p(G)\otimes S)^H.
\]

Moreover, $F_\lambda^p(G)$ has the strong $p$-OAP
by Theorem~\ref{thm:ALRAP}.
Since $S$ is a $p$-operator subspace of $B(\ell^p)$,
it follows that $F_\lambda^p(G)$ has the
$p$-operator approximation property for $S$.
By Theorem~\ref{thm:OAPslice},
it therefore has the slice map property for $S$.

Applying Proposition~\ref{prop:UZ220},
with the trivial action on $F_\lambda^p(G)$, gives
\begin{align*}
(F_\lambda^p(G)\otimes S)^H
&=F(F_\lambda^p(G),S^H,
F_\lambda^p(G)\otimes S)\\
&=F_\lambda^p(G)\otimes S^H.
\end{align*}
Combining the preceding equalities proves the result.
\end{proof}

\subsection{Group extensions}

\begin{thm} \label{thm:extension-ITAP}
Let
\[
1\longrightarrow H\longrightarrow \Gamma
\xrightarrow{\pi} Q\longrightarrow1
\]
be a short exact sequence of countable discrete groups. Let $E$ be an $SQ_p$
space and let $S\subseteq B(E)$ be a $p$-operator space. Suppose that
$Q$ has the $p$-AP and that $H$ has the $p$-ITAP for $S$. Then
$\Gamma$ has the $p$-ITAP for $S$.
\end{thm}

\begin{proof}
Since $Q$ has the $p$-AP,
Corollary~\ref{cor:multiplier-transfer}, applied
to the quotient homomorphism $\pi:\Gamma\to Q$,
gives a net $(u_\alpha)$ of finitely supported
elements of $A_p(Q)$ such that
\[
\widehat m_{u_\alpha}^\pi(x)\longrightarrow x
\]
in norm for every $x\in B_u^p(|\Gamma|,S)$.

Let $T\in B_u^p(|\Gamma|,S)^\Gamma$ and let
$u\in A_p(Q)$ be finitely supported.
For each $q\in\operatorname{supp}(u)$, choose
$r_q\in\Gamma$ such that $\pi(r_q)=q$.
Since $\ker\pi=H$, we have
$\pi^{-1}(q)=Hr_q$, and hence
\[
\widehat m_u^\pi(T)
=
\sum_{q\in\operatorname{supp}(u)}
u(q)M_{Hr_q}(T),
\]
where $M_{Hr_q}$ denotes Schur multiplication by
the kernel
$(s,t)\mapsto\mathbf 1_{Hr_q}(st^{-1})$.

Since $H$ has the $p$-ITAP for $S$,
Proposition~\ref{prop:finite-index} gives
\[
M_{Hr_q}(T)\in F_\lambda^p(\Gamma)\otimes S
\]
for every $q\in\operatorname{supp}(u)$.
Thus
$\widehat m_u^\pi(T)\in F_\lambda^p(\Gamma)\otimes S$.

Applying this to the net $(u_\alpha)$ and passing
to the norm limit yields
$T\in F_\lambda^p(\Gamma)\otimes S$.
The reverse inclusion is immediate.
\end{proof}

\begin{cor} \label{cor:extension-equivalence}
Let
\[
1\longrightarrow H\longrightarrow\Gamma
\longrightarrow Q\longrightarrow1
\]
be a short exact sequence of countable discrete groups, and suppose that $Q$
has the $p$-AP. Let $E$ be an $SQ_p$ space and let
$S\subseteq B(E)$ be a $p$-operator space. Then
\[
\Gamma\text{ has the $p$-ITAP for }S
\quad\Longleftrightarrow\quad
H\text{ has the $p$-ITAP for }S.
\]
Consequently,
\[
\Gamma\text{ has the $p$-operator ITAP}
\quad\Longleftrightarrow\quad
H\text{ has the $p$-operator ITAP},
\]
and
\[
\Gamma\text{ has the scalar $p$-ITAP}
\quad\Longleftrightarrow\quad
H\text{ has the scalar $p$-ITAP}.
\]
In particular, these equivalences hold whenever $Q$ is amenable or
$p$-weakly amenable. If $Q$ has the Haagerup--Kraus approximation
property, then they hold for every $p\in(1,\infty)$.
\end{cor}

\begin{proof}
The equivalence follows from subgroup permanence
and Theorem~\ref{thm:extension-ITAP}.
Taking $S=\mathbb C$ or considering all
$p$-operator spaces $S\subseteq K(\ell^p)$ gives
the corresponding equivalences for the scalar
$p$-ITAP and the $p$-operator ITAP, respectively.
The remaining assertions follow from the
$p$-approximation properties recalled in
Remark~\ref{rem:pAP}.
\end{proof}

\begin{cor} \label{cor:semidirect-ITAP}
Let $G$ and $H$ be countable discrete groups, suppose that $G$ has the $p$-AP,
and let
\[
\alpha:G\longrightarrow\operatorname{Aut}(H)
\]
be an action. Then, for every $SQ_p$ space $E$ and every $p$-operator
space $S\subseteq B(E)$,
\[
H\rtimes_\alpha G\text{ has the $p$-ITAP for }S
\quad\Longleftrightarrow\quad
H\text{ has the $p$-ITAP for }S.
\]
In particular,
\[
G\times H\text{ has the $p$-ITAP for }S
\quad\Longleftrightarrow\quad
H\text{ has the $p$-ITAP for }S.
\]
The analogous equivalences hold for the scalar $p$-ITAP and for the
$p$-operator ITAP.
\end{cor}

\begin{proof}
Apply the preceding corollary to the short exact
sequence
\[
1\longrightarrow H\longrightarrow H\rtimes_\alpha G
\longrightarrow G\longrightarrow 1.
\]
The direct product case corresponds to the trivial
action.
\end{proof}

We next establish the $p$-analogue, for discrete
groups, of the permanence of the Haagerup--Kraus
approximation property under group extensions
\cite[Theorem~1.15]{HK}.
A related permanence result for $M_d$-approximation
properties was established in
\cite[Theorem~1.3]{VergaraMd}.

\begin{thm} \label{cor:pAP-extension}
Let $p\in(1,\infty)$ and let
\[
1\longrightarrow H\longrightarrow\Gamma
\xrightarrow{\pi} Q\longrightarrow 1
\]
be a short exact sequence of discrete groups.
If both $H$ and $Q$ have the $p$-AP, then $\Gamma$
has the $p$-AP.
\end{thm}

\begin{proof}
For $r\in\Gamma$ and $v\in M_{cb}A_p(H)$, define
\[
E_r(v)(g)=
\begin{cases}
v(gr^{-1}),& g\in Hr,\\
0,& g\notin Hr.
\end{cases}
\]
We first show that $E_r$ defines a contraction
from $M_{cb}A_p(H)$ to $M_{cb}A_p(\Gamma)$.

Recall from \cite[Proposition~8.4 and
Theorem~8.6]{Daws10} that the norm of a multiplier
$u\in M_{cb}A_p(\Gamma)$ equals the norm of Schur
multiplication by the kernel
$(s,t)\mapsto u(st^{-1})$ on $B(\ell^p(\Gamma))$,
and conversely every bounded Schur multiplier
of this form arises from an element of
$M_{cb}A_p(\Gamma)$.

With respect to the decomposition of $\Gamma$
into right cosets of $H$, Schur multiplication
by the kernel
$(s,t)\mapsto E_e(v)(st^{-1})$ amounts to retaining
only the diagonal blocks and applying to each
block the Schur multiplier associated with $v$
on $B(\ell^p(H))$.
Consequently,
\[
\|E_e(v)\|_{M_{cb}A_p(\Gamma)}
\leq \|v\|_{M_{cb}A_p(H)}.
\]
For general $r\in\Gamma$, the matrix-entry identity
\[
M_{E_r(v)}(a)
=
M_{E_e(v)}(a\lambda_{r^{-1}})\lambda_r,
\qquad a\in B(\ell^p(\Gamma)),
\]
where $M_w$ denotes Schur multiplication by the
kernel $(s,t)\mapsto w(st^{-1})$, gives
\[
\|E_r(v)\|_{M_{cb}A_p(\Gamma)}
\leq \|v\|_{M_{cb}A_p(H)}.
\]

Moreover, $E_r$ is weak$^*$-continuous with respect
to the canonical preduals.
Indeed, the composition of an evaluation functional
at $g\in\Gamma$ with $E_r$ is either the evaluation
functional at $gr^{-1}\in H$ or the zero functional.
Since the linear spans of the evaluation functionals
are norm dense in the corresponding preduals,
the boundedness of $E_r$ gives the assertion.

Similarly, the pullback map
\[
P:M_{cb}A_p(Q)\longrightarrow M_{cb}A_p(\Gamma),
\qquad P(u)=u\circ\pi,
\]
is contractive by the factorization characterization
in \cite[Proposition~8.4 and Theorem~8.6]{Daws10},
and is weak$^*$-continuous by the same argument
with evaluation functionals.

Let $\mathcal L$ be the weak$^*$-closure in
$M_{cb}A_p(\Gamma)$ of the space of finitely
supported functions on $\Gamma$.
Since $H$ has the $p$-AP, and since $E_r$ is
weak$^*$-continuous and preserves finite support,
we have
\[
\mathbf 1_{Hr}=E_r(1)\in\mathcal L
\]
for every $r\in\Gamma$.

If $u\in A_p(Q)$ is finitely supported, choose
$r_q\in\Gamma$ with $\pi(r_q)=q$ for each
$q\in\operatorname{supp}(u)$.
Then
\[
P(u)=
\sum_{q\in\operatorname{supp}(u)}
u(q)\mathbf 1_{Hr_q}\in\mathcal L.
\]
Since $Q$ has the $p$-AP, the weak$^*$-continuity
of $P$ implies that $1=P(1)\in\mathcal L$.
Thus $\Gamma$ has the $p$-AP.
\end{proof}

In particular, if $\Gamma$ is countable, then under
the hypotheses of the preceding theorem we have
\[
B_u^p(|\Gamma|,S)^\Gamma
=
F_\lambda^p(\Gamma)\otimes S
\]
for every $SQ_p$ space $E$ and every closed
$p$-operator space $S\subseteq B(E)$, by
Theorem~\ref{thm:pAPcoeff}(ii).

\section*{Acknowledgments}

We thank Ignacio Vergara, Zhen Wang, and Jianguo Zhang for their insightful comments in private communication during the course of this work. The author is supported by the Research Start-up Fund of Jilin University (grant number 419080524100).

\bibliographystyle{plain}
\bibliography{mybib}
\end{document}